\documentclass{amsart}
\usepackage{amscd,euscript,color}
\usepackage{amsfonts}
\usepackage{theoremref}
\usepackage{amsmath}
\usepackage{amssymb}
\usepackage{amsthm}
\usepackage{bm}
\usepackage{tikz-cd}
\usepackage{mathtools}
\usepackage[hyphens,spaces,obeyspaces]{url}
\usepackage{hyperref}
\usepackage{graphicx}
\usepackage{enumerate}
\usepackage{comment}
\usepackage{color}
\hypersetup{
colorlinks=false,
linkcolor=black,
filecolor=black,      
urlcolor=black,
}
\newcommand{\cE}{\mathcal{E}}
\newcommand{\cB}{\mathcal{B}_{J_{0}}}
\newcommand{\cBi}{\mathcal{B}_{J_{\infty}}}
\newcommand{\cQuad}{\mathcal{Q}uad}
\newcommand{\ca}{\underline{a}}

\newcommand{\gm}{\mathbb{G}_{m}}
\newcommand{\clz}{lim_{w}}
\newcommand{\cli}{lim_{w'}}
\newcommand{\Address}{{% additional braces for segregating \footnotesize
  \bigskip
  \footnotesize

  Rahul Singh, \textsc{Department of Mathematics, Louisiana State University, Baton Rouge, LA, USA}\par\nopagebreak
  \textit{E-mail address}: \texttt{rahulsingh@lsu.edu}
}}

\theoremstyle{plain}
\newtheorem{theorem}{Theorem}[section]
\newtheorem{proposition}{Proposition}[section]
\newtheorem{corollary}{Corollary}[theorem]
\newtheorem{lemma}[theorem]{Lemma}
\newtheorem{fact}{Fact}[section]
\newtheorem{convention}{Convention}[section]
\newtheorem{notation}{Notation}[section]

\theoremstyle{definition}
\newtheorem{definition}{Definition}

\theoremstyle{remark}
\newtheorem{remark}{Remark}[section]
\newtheorem{example}{Example}
\title{Counting parabolic principal $G$-bundles
with nilpotent sections over $\mathbb{P}^1$}
\author{Rahul Singh}
\date{}
\DeclareMathOperator{\rk}{rk}

\DeclareMathOperator{\Exp}{Exp}
\DeclareMathOperator{\aut}{Aut}

\DeclareMathOperator{\Sym}{Sym}

\DeclareMathOperator{\lie}{Lie}

\DeclareMathOperator{\diag}{diag}
\DeclareMathOperator{\fin}{fin}

\makeatletter
\def\subsubsection{\@startsection{subsubsection}{3}%
\z@{.5\linespacing\@plus.2\linespacing}{-.5em}%
{\smallfont\bfseries}}
\makeatother

\begin{document}
\begin{abstract}
Let $G$ be a split connected reductive group over $\mathbb{F}_q$ and let $\mathbb{P}^1$ be the projective line over $\mathbb{F}_q$. Firstly, we give an explicit formula for the number of $\mathbb{F}_{q}$-rational points of generalized Steinberg varieties of $G$. Secondly, for each principal $G$-bundle over $\mathbb{P}^1$, we give an explicit formula counting the number of triples consisting of parabolic structures at $0$ and $\infty$ and a compatible nilpotent section of the associated adjoint bundle. In the case of $GL_{n}$ we calculate a generating function of such volumes re-deriving a result of Mellit.
\end{abstract}
\maketitle
\renewcommand\abstractname{\textbf{Abstract}}
\tableofcontents
\section{Introduction}
Let $X$ be a smooth projective geometrically connected curve over a finite field $\mathbb{F}_{q}$ with $q$ elements. In a breakthrough paper \cite{Sch}, O.~Schiffmann computed the number of stable Higgs bundles over $X$ of coprime rank and degree when char($\mathbb{F}_{q}$) is sufficiently large (see \cite[Theorem 1.2]{Sch}). In a later paper \cite{MS20} with Mozvogoy, the condition on char($\mathbb{F}_{q}$) was removed. A major step in their calculation is computing the weighted number of vector bundles over $X$ with nilpotent endomorphisms.
%O.~Schiffmann in a breakthrough paper \cite{Sch} computed the number of stable Higgs bundles over $X$ when char($\mathbb{F}_{q}$) is sufficiently large (see \cite[Theorem 1.2]{Sch}). In a later paper \cite{MS20} with Mozvogoy, the condition on char($\mathbb{F}_{q}$) was removed. A major step in their calculation is computing the weighted number of vector bundles over $X$ with nilpotent endomorphisms. 
A.~Mellit in \cite{Mel} has generalized the result of Mozvogoy and Schiffmann to the parabolic case. In particular, Mellit 
 counts vector bundles over $X$ with nilpotent endomorphisms preserving parabolic structures at marked points. An important part of his calculation is the case of $\mathbb{P}^{1}$ and two marked points, this case allows him to relate the count with modified Macdonald polynomials. It is a natural question to generalize Mellit's calculations to arbitrary reductive groups. In this paper, we complete this step, namely, we count the number of principal $G$-bundles over $\mathbb{P}^1$ with nilpotent sections of adjoint bundles compatible with parabolic structures at $0$ and $\infty$ for any split connected reductive group over $\mathbb{F}_{q}$ (Theorem \ref{steinberg} and Theorem \ref{trip}). We note that the paper solves 
 %the geometric part of 
 an important counting problem in the very
first nontrivial case; of course, in view of the original problem of counting semistable parabolic $G$-Higgs bundles over a general curve $X$, the methods used in this paper have a limited scope as it used some rather special features of the
situation (see the explicit description of principal $G$-bundles over $\mathbb{P}^1$ given in Section \ref{basic_principal bundles} and also Remark \ref{twopointsonto}$(ii)$). In principle, if we had some sort of
a product formula as in Proposition \ref{mellit's result}, it would be possible to solve the similar problem
for $\mathbb{P}^1$ with arbitrary number of points. This would pave the way for counting $G$-Higgs
bundles over $\mathbb{P}^1$ with marked points.

In the case of $\mathbb{P}^1$, Mellit uses Hall algebras, which are not easily accesible for a general reductive group. Instead, we use geometric techniques in our proof. We also derive the result in the case of $GL_n$ using our methods.
The counting has two important steps. 
In the first step, we give an explicit formula for the number of points of generalized Steinberg varieties in Theorem \ref{steinberg}. To this end we introduce a coproduct for any reductive group, which might be of independent interest. In the second step, we reduce the problem to counting the number of points of generalized Steinberg varieties using Bialynicki--Birula decomposition in Theorem \ref{trip}. We note that the applicability of Bialynicki--Birula decomposition is not obvious since the schemes that we work with are neither smooth nor projective.

In \cite[Section 7]{Mel}, the generating function of parabolic vector bundles over $X$ with nilpotent endomorphisms preserving parabolic structures at marked points is related to the generating function of the semistable parabolic Higgs bundles in three steps. In the first step, the former function
%of vector bundles over $X$ with nilpotent endomorphisms preserving parabolic structures at marked points 
is related to the generating function of parabolic vector bundles over $X$ with all endomorphisms preserving parabolic structures at marked points via plethystic exponential (see \cite[Section 7.2]{Mel} ). In the next step, the latter count 
%of parabolic vector bundles over $X$ with all endomorphisms preserving parabolic structures at marked points 
is related to the count of parabolic Higgs bundles over $X$ via Serre duality for parabolic vector bundles (see \cite[Section 7.3]{Mel} ). In the last step, the generating function of parabolic Higgs bundles over $X$ is related to the the generating function of semistable parabolic Higgs bundles over $X$ via an integration map (see \cite[Section 7.4]{Mel}). %For a general reductive group $G$, the count of parabolic vector bundles over $X$ with all endomorphisms preserving parabolic structures at marked points can be obtained from the count of parabolic vector bundles over $X$ with nilpotent endomorphisms preserving parabolic structures at marked points upon multiplication by a power of $q$ (see Remark \ref{twopointsonto}$(iii)$). 
To the best of my knowledge, the tools used in the above reductions, that is, plethystic exponential, Serre duality for parabolic vector bundles and the integration map, do not have an analogue for a general reductive group and might be problems of independent interest.
% would complete the algebraic part of the original problem in the case of $\mathbb{P}^1$ and two marked points.  

\hspace*{-5.5mm}
\textbf{Acknowledgements.} 
The author would like to thank his PhD advisor Roman Fedorov for being very patient in answering various questions and his help with editing numerous versions of this paper. 
The author is grateful to Bogdan Ion for several helpful discussions. 
%The author would like to thank Thomas Hales, Kiumars Kaveh and Olivier Schiffmann for asking very interesting questions.
The author is partially supported by NSF DMS grants $1764391$ and
$2001516$. A part of the work was done while visiting the University of Lyon. Finally, the author would like to thank anonymous referees for useful comments and interesting questions.
\section{Preliminaries}\label{section2}

\begin{convention}
$k$ denotes an arbitrary field. When $k$ is fixed, we denote by $\mathbb{P}^{1}$ the projective line over $k$ and by $\gm$ the multiplicative $k$-group $\mathbb{G}_{m,k}$. We denote by $\mathbb{F}_{q}$ the finite field with $q$ elements. For any scheme $X$ over $\mathbb{F}_{q}$, denote by $|X|$ the number of $\mathbb{F}_{q}$-rational points of $X$. 
%We denote by $\gm$ the multiplicative $k$-group $\mathbb{G}_{m,k}$.
\end{convention}
\subsection{Split reductive groups and its Lie algebras}
By an \emph{affine algebraic group} over $k$, we mean a smooth affine $k$-group scheme. A torus over $k$ is said to be \emph{split} if it is isomorphic to $\mathbb{G}^{r}_{m}$ for some $r$. A connected affine algebraic group $G$ over $k$ is said to be \emph{reductive} (\cite[Section 6.46]{Mil1}) if $G_{\overline{k}}$ is reductive.
Recall that a connected reductive group over $k$ is called \emph{split} (\cite[Definition 19.22]{Mil1}) if it contains a maximal torus that is split.

%Let us recall the notion of the Lie algebra of an affine algebraic group over $k$ from \cite[Section 10.6]{Mil1}.
%For an affine algebraic group $G$ over $k$, the tangent space of $G$ at the identity element $e$ is defined as:
%\[
%T_{e,G}:=\ker(G(k[\epsilon])\rightarrow G(k)),
%\]
%where $k[\epsilon]$ is the ring of dual numbers over $k$. Let $I_{G}$ be the augmentation ideal, which is defined to be $\ker(\mathcal{O}(G)\xrightarrow{e^{*}}k)$, where $e^{*}:\mathcal{O}(G)\rightarrow k$ is the co-identity map. One has the following isomorphism
%\[
%T_{e,G}\simeq\text{Hom}_{k-\text{linear}}(I_{G}/I_{G}^{2},k).
%\]
%We define the \emph{Lie algebra} of $G$ to be $\text{Hom}_{k-\text{linear}}(I_{G}/I_{G}^{2},k)$, which 
We will denote the Lie algebra of an affine algebraic group $G$ over $k$ by $\mathfrak{g}$ or sometimes by Lie$(G)$. 
%For the definition of the Lie bracket on $\mathfrak{g}$, we refer to \cite[Section 10.22]{Mil1}

Recall that an element $x\in\mathfrak{g}$ is said to be \emph{nilpotent} if $r(x)$ is nilpotent for every Lie algebra homomorphism $r:\mathfrak{g}\rightarrow \mathfrak{g}\mathfrak{l}(V)$, where $V$ varies over all finite dimensional vector spaces over $k$.

%A torus over $k$ is said to be split if is isomorphic to $\mathbb{G}^{r}_{m}$ for some $r$.
%Recall that a connected reductive group over $k$ is called split if it contains a maximal torus that is split.
%We refer to \cite[Section 10.6]{Mil1} for the notion of the Lie algebra of an affine algebraic group over a field. 
%For an affine algebraic group $G$ over $k$, we will denote its Lie algebra by $\mathfrak{g}$ or sometimes by Lie$(G)$. 
%We say an element $x\in\mathfrak{g}$ is nilpotent if $r(x)$ is nilpotent for every Lie algebra homomorphism $r:\mathfrak{g}\rightarrow \mathfrak{g}\mathfrak{l}(V)$, where $V$ varies over all finite dimensional vector spaces over $k$.
\subsection{Parabolic and Levi $k$-subgroups}\label{parabolic-levi}

Recall that a smooth closed $k$-subgroup $P\subset G$ is parabolic
if the coset space $G/P$ is proper over $k$ (see \cite[Definition 17.15]{Mil1}). Since $G/P$ is quasi-projective over $k$ (see \cite[Theorem 8.44]{Mil1}), we see that for a parabolic $k$-subgroup $P$ of $G$, $G/P$ is projective over $k$.
By a Levi $k$-subgroup of $G$ we mean a Levi factor of a parabolic $k$-subgroup.

In the rest of the paper, $G$ will denote a split connected reductive group over $k$ with a fixed split maximal torus $T$ and a Borel $k$-subgroup $B$ containing $T$ with unipotent radical $U$. Denote by $W$ the Weyl group of $G$ relative to $T$. Further, $X^{*}(T):=\text{Hom}_{k}(T,\gm)$ and $X_{*}(T):=\text{Hom}_{k}(\gm,T)$ will denote the lattices of $k$-characters of $T$ and $k$-cocharacters of $T$ respectively. There is a natural perfect pairing $X^{*}(T)\times X_{*}(T)\rightarrow\mathbb{Z}$, which we denote by $\langle\cdot,\cdot\rangle$. Next, $\Pi\subset\Phi^{+}\subset\Phi\subset X^{*}(T)$ will denote the corresponding simple roots, the positive roots and the root system (see \cite[Chapter 21]{Mil1}).
%$G$ will denote a split connected reductive group over $k$ with a fixed split maximal torus $T$ and a Borel $k$-subgroup $B$ containing $T$. Denote by $W$ the Weyl group of $G$ relative to $T$. Next, $\Pi\subset\Phi^{+}\subset\Phi\subset X^{*}(T)$ will denote the corresponding simple roots, the positive roots and the root system (see \cite[Proposition 11.3.8]{Con1}). Further, $X^{*}(T):=\text{Hom}_{k}(T,\gm)$ and $X_{*}(T):=\text{Hom}_{k}(\gm,T)$ will denote the lattices of $k$-characters of $T$ and $k$-cocharacters of $T$ respectively. 
\subsection{Parametrization of standard parabolic $k$-subgroups}\label{parabolic subgroups}
Let us now recall the  description of standard parabolic $k$-subgroups of $G$ and their Levi factors. Pick $J\subset\Pi$ and 
%(such subgroups are called parabolic). 
let $L_J$ be the the scheme-theoretic centralizer of the identity component of $(\bigcap_{\alpha\in J}\text{Ker }\alpha)_{\text{red}}$. Then $L_{J}$ is a split reductive $k$-group with root system $\Phi_{J}:=\mathbb{Z}J\cap\Phi$ (\cite[Proposition 21.90]{Mil1}). Next, let $U_{J}$ be the $k$-subgroup of $G$ generated by $U_{\alpha}$ (root subgroups), $\alpha\in\Phi^{+}\setminus\Phi_{J}$. Then $P_{J} := L_{J}U_{J}$ is a parabolic $k$-subgroup and $U_{J}$ is the unipotent radical of $P_{J}$ (\cite[Theorem 21.91]{Mil1}). The subgroups $P_{J}$ are called standard parabolic $k$-subgroups and the subgroups $L_{J}$ are called standard Levi $k$-subgroups.
It is known that every parabolic $k$-subgroup is $G(k)$-conjugate to $P_J$ for a unique $J\subset\Pi$ (see \cite[Theorem 21.91 and Theorem 25.8]{Mil1}). It follows that in the case of $G=GL_n$, parabolic $k$-subgroups are precisely the stabilizers of flags in $k^{n}$ and that the Levi $k$-subgroups are precisely the stabilizers of
ordered direct sum decompositions $k^{n}=V_{1}\oplus\ldots\oplus V_{m}$.
%Let us now recall the  description of standard parabolic and Levi $k$-subgroups of $G$. Pick $J\subset\Pi$ and 
%(such subgroups are called parabolic). 
%let $L_J$ be the subgroup generated by $T$ as well as $U_{\alpha}, U_{-\alpha}, \alpha\in J$ (root subgroups). This is a Levi $k$-subgroup (see \cite[Example 12.3.2]{Con1}) and in fact, it is the scheme-theoretic centralizer of the identity component of $(\bigcap_{\alpha\in J}\text{Ker }\alpha)_{\text{red}}$ (recall that Ker $\alpha$ is a $k$-subtorus of $T$) (see \cite[Example 12.3.2]{Con1}). Next, $P_{J} := L_{J}B$ is a parabolic $k$-subgroup (see \cite[Example 12.3.2]{Con1}). The subgroups $P_{J}$ are called standard parabolic $k$-subgroups and the subgroups $L_{J}$ are called standard Levi $k$-subgroups.
%It is known that every parabolic $k$-subgroup is $G(k)$-conjugate to $P_J$ for a unique $J\subset\Pi$ (see \cite[Corollary 11.4.8]{Con1} and the discussion afterwards). It follows that in the case of $G=GL_n$, parabolic $k$-subgroups are precisely the stabilizers of flags in $k^{n}$ (see \cite[Example 11.4.9]{Con1}) and that the Levi $k$-subgroups are precisely the stabilizers of
%ordered direct sum decompositions $k^{n}=V_{1}\oplus\ldots\oplus V_{m}$, (see \cite[Chapter 6]{Con1}).
\begin{notation}
We denote by $X_{+}(T)$ the semilattice of dominant $k$-cocharacters of $T$, i.e, $\lambda\in X_{+}(T)$ if and only if $(\alpha,\lambda)\in\mathbb{Z}_{\geq 0}$ for all $\alpha\in\Phi^{+}$. We note that every $W$-orbit of $X_{*}(T)$ contains exactly one element of $X_{+}(T)$, so we can idenitfy $X_{+}(T)$ with $X_{*}(T)/W$.
\end{notation}
\begin{convention}
%We make the following convention about fibre products of schemes over $k$. 
For any two schemes $X$ and $Y$ over $k$, we will denote $X\times_{k}Y$ by $X\times Y$.
\end{convention}
\subsection{Principal $G$-bundles over $\mathbb{P}^{1}$}\label{basic_principal bundles} 
%Let $Y$ be a scheme over $k$. Let $H$ be a quasi-compact
%connected 
%group scheme over $Y$. Let us review the definition of principal $H$-bundles over $Y$.
%Recall that a $Y$-scheme $\mathcal{P}$ equipped with a right action
%\[
%\mathcal{P}\times H\rightarrow\mathcal{P}
%\]
%of $H$ such that the morphism $\mathcal{P}\rightarrow Y$ is $H$-invariant is called a \emph{principal $H$-bundle} over $Y$, if $\mathcal{P}$ is faithfully flat and quasi-compact over $Y$ and the action is simply transitive, i.e, the natural morphism $\mathcal{P}\times H\rightarrow\mathcal{P}\times_{Y}\mathcal{P}$
%is an isomorphism.
Let $H$ be an affine algebraic group over $k$. Let us recall the construction of associated bundles. 
%Let $M$ be a connected algebraic group over $k$. 
Let $Z$ be a quasi-projective $k$-scheme equipped with a left $H$-action and let $\cE$ be a principal $H$-bundle over a $k$-scheme $S$. 
Then we denote by $\cE\times^{H}Z$ (or sometimes $\cE(Z)$) the associated bundle with fibre type $Z$, which is the following scheme (see \cite[Proposition 3.1]{Fed2}): $\cE\times^{H}Z=(\cE\times Z)/H$ for the right action of $H$ on $\cE\times Z$ given by $h\cdot(e,z)=(e\cdot h,h^{-1}\cdot z)$.

%Let us review the definition of principal $G$-bundles.
%Let $Y$ be a scheme over $k$.
%Recall that a $Y$-scheme $\mathcal{P}$ equipped with a right action
%\[
%\mathcal{P}\times G\rightarrow\mathcal{P}
%\]
%of $G$ such that the morphism $\mathcal{P}\rightarrow Y$ is $G$-invariant is called a principal $G$-bundle over $Y$, if $\mathcal{P}$ is faithfully flat and quasi-compact over $Y$ and the action is simply transitive, i.e, the natural morphism $\mathcal{P}\times G\rightarrow\mathcal{P}\times_{Y}\mathcal{P}$
%is an isomorphism.
%Let us recall the construction of associated bundles. Let $H$ be a connected algebraic group over $k$. Let $Y$ be any quasi-projective $k$-scheme equipped with an $H$-action and let $\cE$ be a principal $H$-bundle over a $k$-scheme $S$. 
%Then we denote by $\cE\times^{H}Y$ the associated bundle with fibre type $Y$, which is the following scheme (see \cite[Proposition 3.1]{Fed2}): $\cE\times^{H}Y=(\cE\times Y)/H$ for the twisted action of $H$ on $\cE\times Y$ given by $h\cdot(e,y)=(e\cdot h,h^{-1}\cdot y)$.
\begin{definition}
Let $H$ and $M$ be affine
%connected
algebraic groups over $k$ and let $\cE$ be a principal $H$-bundle over a $k$-scheme $S$. If $\rho:H\rightarrow M$ is a homomorphism of groups defined over $k$, then the associated bundle $\cE\times^{H}M$ for the action of $H$ on $M$ by left multiplication through $\rho$, is naturally a principal $M$–bundle over $S$. We denote this principal $M$–bundle over $S$ often by $\rho_{*}\cE$ and we say this principal $M$-bundle is obtained from $\cE$ by extension of structure group.
%Let $H$ and $G$ be connected algebraic groups over $k$ and let $\cE$ be a principal $H$-bundle over a $k$-scheme $S$. If $\rho:H\rightarrow G$ is a homomorphism of groups defined over $k$, then the associated bundle $\cE\times^{H}G$ for the action of $H$ on $G$ by left multiplication through $\rho$, is naturally a principal $G$–bundle over $S$. We denote this principal $G$–bundle over $S$ often by $\rho_{*}(\cE)$ and we say this principal $G$-bundle is obtained from $\cE$ by extension of structure group.
\end{definition}
Consider the $\gm$-bundle $\mathcal{O}(1)^\times$ over $\mathbb{P}^{1}$, which is $\mathcal{O}(1)$ minus the zero section. For $\mu\in X_{*}(T)$, define a principal $G$-bundle over $\mathbb P^1$ as:
\[
\cE_{\mu}:=\mu_{*}\mathcal{O}(1)^{\times}
\]
where we view $\mu$ as a morphism $\mu:\mathbb{G}_{m}\rightarrow G$.
Next, every Zariski locally trivial principal $G$-bundle $\cE$ over $\mathbb{P}^{1}$ is isomorphic to exactly one $\cE_{\mu}$, $\mu\in X_{+}(T)$ (see \cite[Theorem 4.2]{Ra}) and in the case $k=\mathbb{F}_{q}$ every principal $G$-bundle is isomorphic to exactly one $\cE_{\mu}$, $\mu\in X_{+}(T)$ (see \cite[Theorem 3.8a)]{Gi} and \cite{La}). Let ad$(\cE_{\mu})$ denote the adjoint vector bundle over $\mathbb{P}^1$ associated to $\cE_\mu$. Note that ad$(\cE_\mu)=\mathcal{O}(1)^{\times}\times^{\mathbb{G}_{m}}\mathfrak{g}$, i.e, it is the quotient of $\mathcal{O}(1)^{\times}\times\mathfrak{g}$ under the action of $\gm$ given by $g\cdot(e,f)=(e\cdot g,\text{Ad}_{\mu(g^{-1})}(f))$, $e\in\mathcal{O}(1)^{\times},f\in\mathfrak{g},g\in\gm$. The sheaf of sections of the adjoint vector bundle ad$(\cE_{\mu})$ form a sheaf of Lie algebras and thus $H^{0}(\mathbb{P}^{1},\text{ad}(\cE_{\mu}))$ has the structure of a Lie algebra. Nilpotent elements of the Lie algebra $H^{0}(\mathbb{P}^{1},\text{ad}(\cE_{\mu}))$ are called nilpotent sections of ad$(\cE_{\mu})$.

%Nilpotent elements of the Lie algebra $H^{0}(\mathbb{P}^{1},\text{ad}(\cE_{\mu}))$ are called nilpotent sections of ad$(\cE_{\mu})$.
\begin{remark}
The statement that every principal $G$-bundle over $\mathbb{P}^1$ is Zariski-locally trivial holds for more general fields.
Recall that a field $k$ is of dimension $\leq 1$ if $Br K=0$ for every algebraic extension $K$ of $k$ (\cite[Proposition 1.5.25]{Poo}). Let $k$ be a perfect field of $\dim k\leq 1$. Then a theorem of Steinberg (see \cite[Theorem 1.9]{Ste}) says that $H^{1}(k,G)=\{*\}$. Therefore by \cite[Theorem 3.8a)]{Gi}, every principal $G$-bundle over $\mathbb{P}^1$ is Zariski-locally trivial in this case.
\end{remark}
\section{Main Results}\label{Main result}
In this section we formulate the main results of this paper. In the special case $G=GL_n$ and $k=\mathbb{F}_{q}$, they give a counting result of Mellit \cite[Section 5.4]{Mel}.
\subsection{Coproduct}\label{coproduct}
Let $H$ be a split connected reductive group over $\mathbb{F}_q$ with a split maximal torus $T_{H}$ and let $B_{H}$ be a Borel $\mathbb{F}_{q}$-subgroup containing $T_H$. Let $\Pi_H\subset X^{*}(T_{H})$ denote the corresponding set of simple roots of $H$. For $J\subset\Pi_H$, let $P_{J}$ denote the standard parabolic $\mathbb{F}_{q}$-subgroup of $H$ and let $L_{J}$ denote the standard Levi factor of $P_J$ (see Section \ref{parabolic subgroups}).
Let $W_H$ denote the Weyl group of $H$ relative to $T_H$ and let $J_{1},J_{2}\subset\Pi_{H}$. We let $W_{i}$ denote the subgroup of $W_{H}$ generated by $s_\alpha$, $\alpha\in J_i$, $i=1,2$. First we make the following conventions for the rest of Section \ref{coproduct} and Section \ref{3.2}:
\begin{convention}
 When the group is clear from the context, we will omit the subscripts and the superscripts $H$, e.g. $T_H$ will be denoted by $T$.   
\end{convention}
We need the following notation:
\begin{notation}
\begin{enumerate}[(i)]
\item
It is known that every double coset in $W_{1}\backslash W/W_{2}$ has a unique minimal length representative $($see \cite[Proposition 2.7.3]{Car}$)$ and we denote this set of representatives by $D^{H}_{J_{1},J_{2}}$. %When we write $w\in W_{1}\backslash W_{H}/W_{2}$, by abuse of notation we will mean that $w$ is any representative of the corresponding double coset.
\item 
For $J\subset\Pi$, denote by $W_J \subset W$ the subgroup generated by $s_{\alpha}, \alpha \in J$.
%here $s_{\alpha}$ denotes the reflection corresponding to $\alpha$. 
For $\mu\in X_{*}(T)$, let $\Pi_{\mu}\subset\Pi$ denote the set of simple roots that are annihilated by $\mu$ and denote by $L_{\mu}$ the identity component of the centralizer of $\mu(\mathbb{G}_{m})$ in $H$. Since $\lie(L_{\mu})=\lie(L_{\Pi_{\mu}})$ (\cite[Theorem 13.33]{Mil1} and Section \ref{parabolic subgroups}), $L_{\mu}=L_{\Pi_{\mu}}$. We note that
 $\Pi_\mu$ is the set of simple roots of $L_{\mu}$ corresponding to $T$ and $B\cap L_{\mu}$. In the special case $H=GL_n$, if $\mu$ is of the form 
\[
t\mapsto\text{diag}(\underbrace{t^{m_1},\ldots,t^{m_1}}_{i_1\text{ times}},\ldots,\underbrace{t^{m_s},\ldots,t^{m_s}}_{i_s\text{ times}}),
\quad m_i\neq m_j \text{ for } i\neq j, m_j\in\mathbb{Z} \text{ for } 1\leq j\leq s,
\]
then $L_\mu=GL_{i_{1}}\times\ldots\times GL_{i_{s}}$.
\end{enumerate}
\end{notation}
Let $\mathcal{P}(\Pi)$ denote the set of subsets of $\Pi$. 
We let $\mathbb{Z}[\mathcal{P}(\Pi)]$ denote the lattice of functions on $\mathcal{P}(\Pi)$ taking values in $\mathbb{Z}$.
For any $f\in\mathbb{Z}[\mathcal{P}(\Pi)]$, define 
\[
\Delta_{H}(f):\mathcal{P}(\Pi)\times\mathcal{P}(\Pi)\rightarrow\mathbb{Z},
\]
which is given by 
\[
\Delta(f)(J_1,J_2):=\sum_{w\in D_{J_{1},J_{2}}}f(J_{1}\cap w\cdot J_{2}).
\]
We will call $\Delta(f)$ as the coproduct of $f$. We have:
\[
\Delta: \mathbb{Z}[\mathcal{P}(\Pi)]\rightarrow \mathbb{Z}[\mathcal{P}(\Pi)]\otimes\mathbb{Z}[\mathcal{P}(\Pi)]\cong\mathbb{Z}[\mathcal{P}(\Pi)\times\mathcal{P}(\Pi)].
\]
\subsection{Generalized Springer and generalized Steinberg varieties}\label{3.2}
For any $J\subset \Pi$, let $Sp_{H}(J)$ denote the \emph{generalized Springer variety} of $H$ with respect to $J$, which is defined as the following scheme of pairs:
\[
Sp_{H}(J):=\{(n,P):P \text{ is $\mathbb{F}_{q}$-conjugate to }P_{J},n \text{ is nilpotent}, n\in \text{Lie}(P)\}.
\]
In particular, $P$ is a parabolic subgroup defined over $\mathbb{F}_{q}$.
For any two subsets $J_1, J_2\subset \Pi_H$, let $St_{H}(J_1,J_2)$ denote the \emph{generalized Steinberg variety} of $H$ with respect to $J_1$ and $J_2$, which is defined as the following scheme of triples:
\[
St_{H}(J_1,J_2):=\{(n,P,Q):P \text{ (resp. }Q)\text{ is $\mathbb{F}_{q}$-conjugate to }P_{J_{1}}\text{ (resp. }P_{J_{2}}), n\text{ is nilpotent}, n\in \text{Lie}(P)\cap \text{Lie}(Q)\}.
\]
In particular, $P$ and $Q$ are parabolic subgroups defined over $\mathbb{F}_{q}$. Observe that $Sp(J)\cong St(\Pi,J)$.

Define
\[
[Sp_H]:\mathcal{P}(\Pi)\rightarrow\mathbb{Z},
\quad J\mapsto|Sp(J)|
\]
and define
\[
[St_H]:\mathcal{P}(\Pi)\times\mathcal{P}(\Pi)\rightarrow\mathbb{Z},\quad (J_1,J_2)\mapsto|St(J_1,J_2)|.
\]
Let $\Phi_{H}$ denote the root system of $H$ with respect to $T$ and let $\Phi_{H}^{+}$ denote the set of positive roots with respect to $B$ and $T$. For $J\subset\Pi$, let $\Phi_{J}$ denote the root system of $L_J$ with respect to $T$ and let $\Phi_{J}^{+}$ denote the set of positive roots with respect to $B\cap L_J$ and $T$.

Before stating the first main theorem, let us recall the notion of rank of an affine algebraic group. 
\begin{notation}
 Let $M$ be an affine algebraic group over $\mathbb{F}_{q}$ and let $\mathfrak{m}$ be the associated Lie algebra. Recall that the rank of $M$ is the dimension of a maximal torus of $M$ or equivalently the dimension of a Cartan subalgebra of $\mathfrak{m}$. We will denote the rank of $M$ by $\rk(M)$ or $\rk(\mathfrak{m})$.
\end{notation}
The following theorem gives an explicit formula for the number of points of generalized Steinberg varieties: 
\begin{theorem}\label{steinberg}
We have
\begin{enumerate}[(i)]
\item $|Sp(J)|=q^{|\Phi_{J}^{+}|+|\Phi^{+}|}\sum_{w\in W/W_J}q^{l(w)}$, where $l(w)$ represents the minimal length of the elements in $wW_J$.
\item $\Delta([Sp])=[St]$.
\end{enumerate}
\end{theorem}
\begin{remark}
Using the equality $|\Phi_{J}^{+}|+|\Phi^{+}|=\dim(P_{J})-\rk(P_{J})$, the formula in $(i)$ can also be written as:
\[
|Sp(J)|=q^{\dim(P_{J})-\rk(P_{J})}\sum_{w\in W/W_J}q^{l(w)}
\]
\end{remark}
We give the proof of Theorem \ref{coproduct} in Section \ref{Steinberg Varieties}.
\subsection{Stratification of triples}\label{main theorem}
\begin{definition}
Fix a $k$-rational point $x$ of $\mathbb{P}^1$. For $J\subset\Pi$, a $\emph{parabolic structure}$ on a principal $G$-bundle $\cE$ over $\mathbb{P}^1$ at $x$ of type $J$ is a choice of a $k$-rational point $\mathfrak{p}_{x}$ of $\cE_{x}/P_{J}$ where $\cE_{x}$ is the fiber of $\cE$ at $x$.
\end{definition}
Let $\mu\in X_{+}(T)$ and $J_{0},J_{\infty}\subset\Pi$, define $\mathcal{T}rip_{\mu}(J_{0},J_{\infty})$ to be the scheme parameterizing triples $(\mathfrak{p}_{0},\mathfrak{p}_{\infty},\Psi)$ such that $\Psi$ is a nilpotent section of ad$(\cE_{\mu})$, $\mathfrak{p}_{0}$ (resp. $\mathfrak{p}_{\infty}$) is a parabolic structure at $0$ (resp. $\infty$) of type $J_{0}$ (resp. $J_{\infty}$) and $\Psi_{0}\in \mathfrak{p}_{0}$,  $\Psi_{\infty}\in \mathfrak{p}_{\infty}$. We note that $\mathcal{T}rip_{\mu}(J_{0},J_{\infty})$ is a scheme because it is the closed subscheme of $\cE_{0}/P_{J_{0}}\times\cE_{\infty}/P_{J_{\infty}}\times H^{0}(\mathbb{P}^{1},\text{ad}(\cE_{\mu}))$ given by three closed conditions which are: $\Psi$ is nilpotent, $\Psi_{0}\in \mathfrak{p}_{0}$,  $\Psi_{\infty}\in \mathfrak{p}_{\infty}$.

Now let us explain the meaning of $\Psi_{x}\in\mathfrak{p}_{x}$, $x=0,\infty$ in the definition of $\mathcal{T}rip_{\mu}(J_{0},J_{\infty})$. 
For $x=0,\infty$, we view $(\cE_{\mu})_{x}$ as a principal $G$-bundle over the point $x$ and we let Aut$((\cE_{\mu})_x)$ denote the $k$-group scheme 
whose $R$-valued points are the principal $G\times Spec(R)$-bundle automorphisms of $(\cE_{\mu})_{x}\times Spec(R)$. 
%By Lang's theorem (see \cite{La}), 
Since $\cE_{\mu}$ is a pushforward of the $\mathbb{G}_{m}$-bundle $\mathcal{O}(1)^{\times}$, $(\cE_{\mu})_{x}$ is a trivial principal $G$-bundle over the point $x$ and therefore $\aut((\cE_{\mu})_{x})$ can be non-canonically identified with $G$. Now, $\aut((\cE_{\mu})_{x})$ acts on $(\cE_{\mu})_{x}/P_{J_{x}}$ and the stabilizer of $\mathfrak{p}_{x}$ is a parabolic subgroup of Aut$((\cE_{\mu})_x)$.
When we write $\Psi_{x}\in\mathfrak{p}_{x}$, we view $\mathfrak{p}_{x}$ as the Lie algebra of this stabilizer. This is a parabolic subalgebra of Lie(Aut$((\cE_{\mu})_{x}))=\text{ad}(\cE_{\mu})_{x}$.

%Since $\mathcal{O}(1)^{\times}$ is a principal $\mathbb{G}_{m}$-bundle over $\mathbb{P}^1$, 
Next, we describe a $\gm$-action on $\mathcal{T}rip_{\mu}(J_{0},J_{\infty})$ coming from the torsor structure on $\cE_{\mu}$. First note that $\mathbb{G}_{m}$ acts on $\cE_{\mu}=(\mathcal{O}(1)^{\times}\times G)/\mathbb{G}_{m}$ by acting on the first component from the left. This gives a $\mathbb{G}_{m}$-action on the parabolic structures and on ad$(\cE_{\mu})$, which gives a $\mathbb{G}_{m}$-action on $H^{0}(\mathbb{P}^{1},\text{ad}(\cE_{\mu}))$. On combining these actions, we get a $\mathbb{G}_{m}$-action:
\begin{equation}\label{action on trip}
    \gm\curvearrowright\mathcal{T}rip_{\mu}(J_{0},J_{\infty}).
\end{equation}
In this paper, in the case when $k=\mathbb{F}_{q}$ we want to count the number of $\mathbb{F}_{q}$-points of $\mathcal{T}rip_{\mu}(J_{0},J_{\infty})$ for each $\mu\in X_{+}(T),J_{0},J_{\infty}\subset\Pi$. For this, we would like to apply the Bialynicki--Birula decomposition to $\mathcal{T}rip_{\mu}(J_{0},J_{\infty})$ with respect to the $\gm$-action \eqref{action on trip}. Note that it is not immediate in this case because $\mathcal{T}rip_{\mu}(J_{0},J_{\infty})$ is neither smooth nor projective in general but nevertheless we will prove below Theorem \ref{trip}, which allows to reduce counting $|\mathcal{T}rip_{\mu}(J_{0},J_{\infty})|$ to counting points of the generalized Steinberg varieties.
\begin{notation}
Let $X$ be a scheme over $k$
%$\mathbb{F}_{q}$ 
and let $H$ be an affine algebraic group over $k$
%$\mathbb{F}_{q}$ 
acting on $X$. We will denote the fixed point locus of this action by $X^H$.
\end{notation}
\begin{theorem}\label{trip}
%Keep notations as above. 
Let $\mathbb{G}_{m}$ act on $\mathcal{T}rip_{\mu}(J_{0},J_{\infty})$ as in $\eqref{action on trip}$. Then there exists a stratification of $\mathcal{T}rip_{\mu}(J_{0},J_{\infty})$ by locally closed subsets as:
\[
\mathcal{T}rip_{\mu}(J_{0},J_{\infty})=\displaystyle\bigsqcup_{\substack{w\in W_{\Pi_{\mu}}\backslash W/W_{J_{0}}
\\
w'\in W_{\Pi_{\mu}}\backslash W/W_{J_{\infty}}}}\mathcal{T}rip_{\mu}(J_{0},J_{\infty})_{w,w'}^{+}
\]
and a decomposition of $\mathcal{T}rip_{\mu}(J_{0},J_{\infty})^{\mathbb G_{m}}$ as:
\[
\mathcal{T}rip_{\mu}(J_{0},J_{\infty})^{\mathbb G_{m}}=\displaystyle\bigsqcup_{\substack{w\in W_{\Pi_{\mu}}\backslash W/W_{J_{0}}
\\
w'\in W_{\Pi_{\mu}}\backslash W/W_{J_{\infty}}}} \mathcal{T}rip_{\mu}(J_{0},J_{\infty})_{w,w'}^{\gm},
\]
where $\mathcal{T}rip_{\mu}(J_{0},J_{\infty})_{w,w'}^{\gm}$ are the connected components of $\mathcal{T}rip_{\mu}(J_{0},J_{\infty})^{\mathbb G_{m}}$ 
with morphisms
\[
\mathcal{T}rip_{\mu}(J_{0},J_{\infty})_{w,w'}^{+}\rightarrow\mathcal{T}rip_{\mu}(J_{0},J_{\infty})_{w,w'}^{\gm},
\]
which are given by the limit map as $t\rightarrow 0$ and are
affine fibrations for $w\in W_{\Pi_{\mu}}\backslash W/W_{J_{0}}, w'\in W_{\Pi_{\mu}}\backslash W/W_{J_{\infty}}$ of relative dimensions $\dim(\aut(\cE_{\mu}))-\dim(L_{\mu})$. Moreover, the schemes $\mathcal{T}rip_{\mu}(J_{0},J_{\infty})_{w,w'}^{\gm}$ are isomorphic to the generalized Steinberg varieties $St_{L_{\mu}}(\Pi_{\mu}\cap w\cdot J_{0},\Pi_{\mu}\cap w'\cdot J_{\infty})$, $w\in D^{G}_{\Pi_{\mu},J_{0}}, w'\in D^{G}_{\Pi_{\mu},J_{\infty}}$. %Moreover, the schemes $\mathcal{T}rip_{\mu}(J_{0},J_{\infty})_{w,w'}^{\gm}, w\in D^{G}_{\Pi_{\mu},J_{0}}, w'\in D^{G}_{\Pi_{\mu},J_{\infty}}$ are isomorphic to the generalized Steinberg varieties $St_{L_{\mu}}(\Pi_{\mu}\cap w\cdot J_{0},\Pi_{\mu}\cap w'\cdot J_{\infty})$ defined in Section \ref{coproduct}.
\end{theorem}
\hspace*{-5.5mm}
The proof of Theorem \ref{trip} will be given in Section \ref{formula}. 

Upto this point, the base field $k$ in Theorem \ref{trip} was arbitrary. Now let $k=\mathbb{F}_{q}$. For $\mu\in X_{+}(T)$, define
$
\pi_{\mu}:\mathbb{Z}[\mathcal{P}(\Pi_{\mu})]\rightarrow \mathbb{Z}[\mathcal{P}(\Pi)]
$
as:
\[
\pi_{\mu}(f)(J):=\sum_{w\in D^{G}_{\Pi_{\mu},J}} f(\Pi_{\mu}\cap w\cdot J),
\quad f\in\mathbb{Z}[\mathcal{P}(\Pi_{\mu})]
\] 
and define $[\mathcal{T}rip_\mu]:\mathcal{P}(\Pi)\times\mathcal{P}(\Pi)\rightarrow\mathbb{Z}$ as:
\[
[\mathcal{T}rip_\mu](J_{0},J_{\infty}):=|\mathcal{T}rip_{\mu}(J_{0},J_{\infty})|
\]
As an easy corollary of Theorem \ref{trip}, we get:
\begin{corollary}\label{final}
Keeping the above notations, we have:
\[
[\mathcal{T}rip_\mu]=q^{\dim(\aut(\cE_{\mu}))-\dim(L_{\mu})}(\pi_{\mu}\otimes\pi_{\mu})([St_{L_{\mu}}]).
\]
\end{corollary}
\begin{proof}
Let $J_{0},J_{\infty}\subset\Pi$. From Theorem \ref{trip}, we have
\[
[\mathcal{T}rip_\mu](J_{0},J_{\infty})=\displaystyle\sum_{\substack{w\in W_{\Pi_{\mu}}\backslash W/W_{J_{0}}
\\
w'\in W_{\Pi_{\mu}}\backslash W/W_{J_{\infty}}}}|\mathcal{T}rip_{\mu}(J_{0},J_{\infty})_{w,w'}^{+}|=\displaystyle\sum_{\substack{w\in W_{\Pi_{\mu}}\backslash W/W_{J_{0}}
\\
w'\in W_{\Pi_{\mu}}\backslash W/W_{J_{\infty}}}}q^{\dim(\aut(\cE_{\mu}))-\dim(L_{\mu})}|\mathcal{T}rip_{\mu}(J_{0},J_{\infty})_{w,w'}^{\gm}|.
\]
Since the schemes $\mathcal{T}rip_{\mu}(J_{0},J_{\infty})_{w,w'}^{\gm}$ are isomorphic to the generalized Steinberg varieties $St_{L_{\mu}}(\Pi_{\mu}\cap w\cdot J_{0},\Pi_{\mu}\cap w'\cdot J_{\infty})$, $w\in D^{G}_{\Pi_{\mu},J_{0}}, w'\in D^{G}_{\Pi_{\mu},J_{\infty}}$
%the schemes $\mathcal{T}rip_{\mu}(J_{0},J_{\infty})_{w,w'}^{\gm}$ are isomorphic to the generalized Steinberg varieties $St_{L_{\mu}}(\Pi_{\mu}\cap w\cdot J_{0},\Pi_{\mu}\cap w'\cdot J_{\infty})$, $w\in D^{G}_{\Pi_{\mu},J_{0}}, w'\in D^{G}_{\Pi_{\mu},J_{\infty}}$
(see Theorem \ref{trip}), we have
%Since the schemes $\mathcal{T}rip_{\mu}(J_{0},J_{\infty})_{w,w'}^{\gm}$ are isomorphic to the generalized Steinberg varieties $St_{L_{\mu}}(\Pi_{\mu}\cap w\cdot J_{0},\Pi_{\mu}\cap w'\cdot J_{\infty})$, $w\in D^{G}_{\Pi_{\mu},J_{0}}, w'\in D^{G}_{\Pi_{\mu},J_{\infty}}$ (see Theorem \ref{trip}), we have
\[
[\mathcal{T}rip_\mu](J_{0},J_{\infty})=q^{\dim(\aut(\cE_{\mu}))-\dim(L_{\mu})}\displaystyle\sum_{\substack{w\in D^{G}_{\Pi_{\mu},J_{0}}
\\
w'\in D^{G}_{\Pi_{\mu},J_{\infty}}}}|St_{L_{\mu}}(\Pi_{\mu}\cap w\cdot J_{0},\Pi_{\mu}\cap w'\cdot J_{\infty})|.
\]
Now the corollary follows from the definition of $\pi_{\mu}$.
\end{proof}
\begin{remark}\label{pi=delta}
\begin{enumerate}[(i)]
\item For deducing Corollary \ref{final} from Theorem \ref{trip}, it is crucial that all fibers of the morphism $\mathcal{T}rip_{\mu}(J_{0},J_{\infty})_{w,w'}^{+}\rightarrow\mathcal{T}rip_{\mu}(J_{0},J_{\infty})_{w,w'}^{\gm}$ have the same dimension.
   %\item Since we know $|Sp_{L_{\mu}}(J)|$ for $\mu\in X_{-}(T),J\subset\Pi_\mu$ explicitly (see Theorem \ref{trip}), Corollary \ref{final} gives an explicit formula for $[\mathcal{T}rip_{\mu}]$.
   \item
   Notice that $\pi_{\mu}$ is an instance of $\Delta_{G}$. 
   More precisely, let $f\in\mathbb{Z}[\mathcal{P}(\Pi_{\mu})]$ and let $\Tilde{f}$ be any extension of $f$ to $\mathcal{P}(\Pi)$, i.e, $\Tilde{f}\in\mathbb{Z}[\mathcal{P}(\Pi)]$ and $\Tilde{f}_{|_{\mathcal{P}(\Pi_{\mu})}}=f$. Then we have $\pi_{\mu}(f)=\Delta_{G}(\Tilde{f})(\Pi_{\mu},\cdot)$.
\end{enumerate}
\end{remark}
More explicitly, we have the following corollary.
\begin{corollary}\label{explicit}
%Keep notations as above. Then $|\mathcal{T}rip_\mu(J_{0},J_{\infty})|$ is equal to
%\[
%q^{\sum_{\langle\alpha,\mu\rangle<0}\big(-\langle\alpha,\mu\rangle+1\big)}\displaystyle\sum_{\substack{w\in D^{G}_{\Pi_{\mu},J_{0}}
%\\
%w'\in D^{G}_{\Pi_{\mu},J_{\infty}}}}\sum_{w''\in D^{L_{\mu}}_{\Pi_{\mu}\cap w\cdot J_{0},\Pi_{\mu}\cap w'\cdot J_{\infty}}}q^{|\Phi_{\Pi_{\mu}\cap w\cdot J_{0}\cap w''\cdot(\Pi_{\mu}\cap w'\cdot J_{\infty})}^{+}|+|\Phi_{\Pi_{\mu}}^{+}|}
%\sum_{w'''\in D^{L_{\mu}}_{\O,\Pi_{\mu}\cap w\cdot J_{0}\cap w''\cdot(\Pi_{\mu}\cap w'\cdot J_{\infty})}}q^{l(w''')},
%\] 
%\[
%|\mathcal{T}rip_\mu(J_{0},J_{\infty})|=|L_{\mu}|q^{\dim(\aut(\cE_{\mu}))-\dim(L_{\mu})}\displaystyle\sum_{\substack{w\in D^{G}_{\Pi_{\mu},J_{0}}
%\\
%w'\in D^{G}_{\Pi_{\mu},J_{\infty}}}}\sum_{w''\in D^{L_{\mu}}_{\Pi_{\mu}\cap w\cdot J_{0},\Pi_{\mu}\cap w'\cdot J_{\infty}}}\frac{q^{|\Phi_{\Pi_{\mu}\cap w\cdot J_{0}\cap w''\cdot(\Pi_{\mu}\cap w'\cdot J_{\infty})}|}}{|L_{\Pi_{\mu}\cap w\cdot J_{0}\cap w''\cdot(\Pi_{\mu}\cap w'\cdot J_{\infty})}|},
%\] 
%where $\Phi_{\Pi_{\mu}\cap w\cdot J_{0}\cap w''\cdot(\Pi_{\mu}\cap w'\cdot J_{\infty})}$ is the root system of $L_{\Pi_{\mu}\cap w\cdot J_{0}\cap w''\cdot(\Pi_{\mu}\cap w'\cdot J_{\infty})}$ with respect to $T$.
%Keep notations as above. 
We have $|\mathcal{T}rip_\mu(J_{0},J_{\infty})|$ is equal to
\[
q^{|\Phi_{\Pi_{\mu}}^{+}|+\sum_{\langle\alpha,\mu\rangle>0}\big(\langle\alpha,\mu\rangle+1\big)}\displaystyle\sum_{\substack{w\in D^{G}_{\Pi_{\mu},J_{0}}
\\
w'\in D^{G}_{\Pi_{\mu},J_{\infty}}}}\sum_{w''\in D^{L_{\mu}}_{\Pi_{\mu}\cap w\cdot J_{0},\Pi_{\mu}\cap w'\cdot J_{\infty}}}q^{\big|\Phi_{\Pi_{\mu}\cap w\cdot J_{0}\cap w''\cdot(\Pi_{\mu}\cap w'\cdot J_{\infty})}^{+}\big|}A(\mu,w,w',w'';q),
\] 
%\[
%|\mathcal{T}rip_\mu(J_{0},J_{\infty})|=|L_{\mu}|q^{\dim(\aut(\cE_{\mu}))-\dim(L_{\mu})}\displaystyle\sum_{\substack{w\in D^{G}_{\Pi_{\mu},J_{0}}
%\\
%w'\in D^{G}_{\Pi_{\mu},J_{\infty}}}}\sum_{w''\in D^{L_{\mu}}_{\Pi_{\mu}\cap w\cdot J_{0},\Pi_{\mu}\cap w'\cdot J_{\infty}}}\frac{q^{|\Phi_{\Pi_{\mu}\cap w\cdot J_{0}\cap w''\cdot(\Pi_{\mu}\cap w'\cdot J_{\infty})}|}}{|L_{\Pi_{\mu}\cap w\cdot J_{0}\cap w''\cdot(\Pi_{\mu}\cap w'\cdot J_{\infty})}|},
%\] 
where $\Phi_{\Pi_{\mu}\cap w\cdot J_{0}\cap w''\cdot(\Pi_{\mu}\cap w'\cdot J_{\infty})}$ is the root system of $L_{\Pi_{\mu}\cap w\cdot J_{0}\cap w''\cdot(\Pi_{\mu}\cap w'\cdot J_{\infty})}$ with respect to $T$ and 
\[
A(\mu,w,w',w'';q)=
\sum_{w'''\in D^{L_{\mu}}_{\emptyset,\Pi_{\mu}\cap w\cdot J_{0}\cap w''\cdot(\Pi_{\mu}\cap w'\cdot J_{\infty})}}q^{l(w''')}.
\]
In particular, we see that $|\mathcal{T}rip_\mu(J_{0},J_{\infty})|$ is a polynomial in $q$ with non-negative integer coefficients.
\end{corollary}
\begin{remark}
\begin{enumerate}[(i)]
    \item  Using the equality $|\Phi_{\Pi_{\mu}}^{+}|+\sum_{\langle\alpha,\mu\rangle>0}\big(\langle\alpha,\mu\rangle+1\big)=|\Phi|/2+\langle2\rho,\mu\rangle$, we see that $|\mathcal{T}rip_\mu(J_{0},J_{\infty})|$ is equal to
    \[
q^{|\Phi|/2+\langle2\rho,\mu\rangle}\displaystyle\sum_{\substack{w\in D^{G}_{\Pi_{\mu},J_{0}}
\\
w'\in D^{G}_{\Pi_{\mu},J_{\infty}}}}\sum_{w''\in D^{L_{\mu}}_{\Pi_{\mu}\cap w\cdot J_{0},\Pi_{\mu}\cap w'\cdot J_{\infty}}}q^{\big|\Phi_{\Pi_{\mu}\cap w\cdot J_{0}\cap w''\cdot(\Pi_{\mu}\cap w'\cdot J_{\infty})}^{+}\big|}A(\mu,w,w',w'';q),
\]
where $\rho$ is the half sum of positive roots.
\item When $\mu$ is generic, i.e, $\langle\alpha,\mu\rangle\neq 0$ $\forall\alpha\in\Phi$, the expression in the above corollary simplifies to
\[
q^{|\Phi|/2+\langle2\rho,\mu\rangle}\frac{|W|^{2}}{|W_{J_{0}}||W_{J_{\infty}}|}.
\]
\end{enumerate}
\end{remark}
\begin{example}
    Let $G=GL_{3}$ and let $B$ (resp. $T$) be the subgroup of $G$ consisting of upper triangular (resp. diagonal) matrices. Let $\Phi$ denote the set of roots of $G$ with respect to $T$ and let $\Pi$ denote the set of simple roots with respect to $(B,T)$. Under the standard identification, we have $X_{*}(T)=\mathbb{Z}^{3}$, $W=S_{3}$, $\mu=(a,b,c)$, $a\geq b\geq c$ and $\Pi=\{e_{1}-e_{2},e_{2}-e_{3}\}$. We look at the following choices in order to illustrate Corollary \ref{explicit}:
    \begin{enumerate}[(i)]
        \item $a>b=c$, $J_{0}=\{e_{1}-e_{2}\}$, $J_{\infty}=\{e_{1}-e_{2}\}$: $|\mathcal{T}rip_\mu(J_{0},J_{\infty})|=q^{2(a-c)+3}(4+5q)$.
        \item $a>b=c$, $J_{0}=\{e_{1}-e_{2}\}$, $J_{\infty}=\{e_{2}-e_{3}\}$: $|\mathcal{T}rip_\mu(J_{0},J_{\infty})|=q^{2(a-c)+3}(4+5q)$.
        \item $a>b=c$, $J_{0}=\Pi$, $J_{\infty}=\{e_{1}-e_{2}\}$: $|\mathcal{T}rip_\mu(J_{0},J_{\infty})|=q^{2(a-c)+3}(2q+1)$.
        \item $a=b=c$, $J_{0}=\{e_{1}-e_{2}\}$, $J_{\infty}=\{e_{1}-e_{2}\}$: $|\mathcal{T}rip_\mu(J_{0},J_{\infty})|=q^{3}(1+3q+3q^{2}+2q^{3})$.
        \item $a=b=c$, $J_{0}=\{e_{1}-e_{2}\}$, $J_{\infty}=\{e_{2}-e_{3}\}$: $|\mathcal{T}rip_\mu(J_{0},J_{\infty})|=q^{3}(1+3q+3q^{2}+2q^{3})$.
        \item $a=b=c$, $J_{0}=\Pi$, $J_{\infty}=\{e_{1}-e_{2}\}$: $|\mathcal{T}rip_\mu(J_{0},J_{\infty})|=q^{4}(1+q+q^{2})$.
    \end{enumerate}
\end{example}
%\begin{proof}
%By Theorem \ref{steinberg}$(ii)$ and Theorem \ref{trip}, we get
%\begin{equation}\label{last page}
%[\mathcal{T}rip_\mu](J_{0},J_{\infty})=q^{\dim(\aut(\cE_{\mu}))-\dim(L_{\mu})}\displaystyle\sum_{\substack{w\in D^{G}_{\Pi_{\mu},J_{0}}
%\\
%w'\in D^{G}_{\Pi_{\mu},J_{\infty}}}}\Delta_{L_{\mu}}([Sp_{L_{\mu}}])(\Pi_{\mu}\cap w\cdot J_{0},\Pi_{\mu}\cap w'\cdot J_{\infty}).
%\end{equation}
%Now the corollary follows from the defintion of $\Delta_{L_{\mu}}$,Theorem \ref{steinberg}$(i)$ and Fact \ref{aut of emu}.
To prove Corollary \ref{explicit}, we need the following result (see \cite[Proposition 5.2]{Ra}), which describes $\aut(\cE_{\mu})$ as a scheme:
\begin{fact}\label{aut of emu}
Let $\cE_{\mu}$ be as above. Then $\aut(\cE_{\mu})$ is isomorphic as a scheme to
\[
L_{\mu}\times \prod_{\alpha\in\Phi:\langle\alpha,\mu\rangle>0}H^{0}(\mathbb{P}^{1},\mathcal{O}(\langle\alpha,\mu\rangle)).
\]
\end{fact}
\begin{proof}(of Corollary \ref{explicit})
By Theorem \ref{steinberg}$(ii)$ and Corollary \ref{final}, we get
\begin{equation}\label{last page}
[\mathcal{T}rip_\mu](J_{0},J_{\infty})
=
q^{\dim(\aut(\cE_{\mu}))-\dim(L_{\mu})}\displaystyle\sum_{\substack{w\in D^{G}_{\Pi_{\mu},J_{0}}
\\
w'\in D^{G}_{\Pi_{\mu},J_{\infty}}}}\Delta_{L_{\mu}}([Sp_{L_{\mu}}])(\Pi_{\mu}\cap w\cdot J_{0},\Pi_{\mu}\cap w'\cdot J_{\infty}).
\end{equation}
Using the definition of $\Delta_{L_{\mu}}$, $[\mathcal{T}rip_\mu](J_{0},J_{\infty})$ is equal to
\[
q^{\dim(\aut(\cE_{\mu}))-\dim(L_{\mu})}\displaystyle\sum_{\substack{w\in D^{G}_{\Pi_{\mu},J_{0}}
\\
w'\in D^{G}_{\Pi_{\mu},J_{\infty}}}}
\sum_{w''\in D^{L_{\mu}}_{\Pi_{\mu}\cap w\cdot J_{0},\Pi_{\mu}\cap w'\cdot J_{\infty}}}
|Sp_{L_{\mu}}(\Pi_{\mu}\cap w\cdot J_{0}\cap w''\cdot(\Pi_{\mu}\cap w'\cdot J_{\infty}))|.
\]
Now the corollary follows from Theorem \ref{steinberg}$(i)$ and Fact \ref{aut of emu}.
\end{proof}
It follows from definitions that $\mathcal{T}rip_0(J_0,J_\infty)=St_G(J_0, J_\infty)$ and $\mathcal{T}rip_0(\Pi,\Pi)=\mathcal{N}(\mathfrak{g})$, the nilpotent cone of $\mathfrak{g}$. In particular, we see that even in the trivial case $\mu=0, J_{0}=J_{\infty}=\Pi$, $\mathcal{T}rip_{\mu}(J_{0},J_{\infty})$ is neither smooth nor projective. 

We note the following corollary.
\begin{corollary}\label{contribution trivial bundle}
%Keep notations as above and 
Assume that $\mu\in X_{+}(T)$ is a central cocharacter. Then $[\mathcal{T}rip_{\mu}]=[St_{G}]$.
\end{corollary}
\begin{proof}
It follows from Corollary \ref{explicit} that $[\mathcal{T}rip_{\mu}]=[\mathcal{T}rip_{0}]$.
\end{proof}
\section{Generalized Steinberg Varieties}\label{Steinberg Varieties}
In this section, we give a proof of Theorem \ref{steinberg}. Recall that for any scheme $X$ over $\mathbb{F}_{q}$, we denote the number of $\mathbb{F}_{q}$-rational points of $X$ by $|X|$. 

%We now prove a simple lemma that will be used several times in the paper:
\begin{notation}
Let $M$ be an algebraic group over $\mathbb{F}_{q}$ and let $\mathfrak{m}$ be the associated Lie algebra. We will denote the nilpotent cone of $\mathfrak{m}$ by $\mathcal{N}(\mathfrak{m})$. 
\end{notation}
\subsection{Proof of Theorem \ref{steinberg}$(i)$}
Let $H$ be a split reductive group over $\mathbb{F}_{q}$ with a split maximal torus $T_{H}$ and let $B_{H}$ be a Borel $\mathbb{F}_{q}$-subgroup containing $T_H$. Let $\Pi_{H}\subset X^{*}(T_{H})$ denote the corresponding set of simple roots of $H$. 
Let $W_H$ denote the Weyl group of $H$ relative to $T_H$. 
For any $J\subset \Pi_H$, let $P_J$ be the corresponding standard parabolic $\mathbb{F}_{q}$-subgroup of $H$. Let $L_{J}$ and $U_{J}$ be the Levi factor and the unipotent radical of $P_{J}$, respectively and let $W_J$ be the corresponding subgroup of $W_H$. 

The number of points of the generalized Springer variety of $H$ corresponding to $J\subset\Pi_H$ is given by
\begin{equation}\label{intermediate}
|Sp_{H}(J)|=\frac{|H|}{|P_{J}|}|\mathcal{N}(\text{Lie}(P_J)|=\frac{|H|}{|P_{J}|}q^{\text{dim}(P_{J})-\rk(P_{J})},
\end{equation}
where the first equality holds because the normalizer of $P_J$ in $H(\mathbb{F}_{q})$ is $P_{J}(\mathbb{F}_{q})$ and the fact that if $P$ is a parabolic subgroup of $G$ conjugate over $\mathbb{F}_{q}$ to $P_{J}$, then $\mathcal{N}(\lie(P))\cong\mathcal{N}(\lie(P_{J}))$. The second equality follows from $|\mathcal{N}(\text{Lie}(P_J))|=q^{\dim(P_{J})-\rk(P_{J})}$.

Since $H/P_{J}$ has a stratification by locally closed subsets as $\bigsqcup_{{w\in W_{H}/W_J}}\mathbb{A}^{l(w)}$ (see \cite[Proposition 3.16]{BT}), where $l(w)$ represents the minimal length of the elements in $wW_J$, using $|H/P_{J}|=|H|/|P_{J}|$ we get
\[
|Sp_{H}(J)|=q^{|\Phi_{J}^{+}|+|\Phi_{H}^{+}|}\sum_{w\in W_{H}/W_J}q^{l(w)}.
\]
This finishes the proof of part ($i$) of Theorem \ref{steinberg}.
\subsection{Proof of Theorem \ref{steinberg}$(ii)$}
In the proof of part $(ii)$ of Theorem \ref{steinberg}, we will need another formula for $|Sp_{H}(J)|$, which we now give. Let $U_J$ denote the unipotent radical of $P_J$. Then, we have $P_J\cong L_J\times U_J$ as schemes over $\mathbb{F}_{q}$ and $|U_J|=q^{\text{dim}(U_J)}$, (see \cite[Corollary 14.2.7]{Spr}). Using these two facts, $\eqref{intermediate}$ gives
\begin{equation}\label{for GLn}
|Sp_{H}(J)|=\frac{|H|}{|L_{J}|} q^{\text{dim}(L_{J})-\rk(L_{J})}.
\end{equation}
Now $|\mathcal{N}(\text{Lie}(L_J))|=q^{\dim(L_{J})-\rk(L_{J})}$ gives
\begin{equation}\label{levi1}
    |Sp_{H}(J)|=|H| \frac{|\mathcal{N}(\text{Lie}(L_J))|}{|L_J|}.
\end{equation}
For any $J_{i}\subset \Pi_H$, let $P_i:=P_{J_{i}}$ be the corresponding standard parabolic $\mathbb{F}_{q}$-subgroup of $H$, $i=1,2$. Let $L_{i}:=L_{J_{i}}$ and $U_{i}:=U_{J_{i}}$ be the Levi factor and the unipotent radical of $P_{i}$ respectively, and let $W_i:=W_{J_{i}}$ be corresponding subgroup of $W_H$, $i=1,2$. Consider the natural action of $H(\mathbb{F}_{q})$ on $St_{H}(J_1,J_2)$. Since the normalizer of $P_{1}$ in $H(\mathbb{F}_{q})$ is $P_{1}(\mathbb{F}_{q})$, the number of points of $St_{H}(J_1,J_2)$ is given by
%the generalized Steinberg variety of $H$ corresponding to $J_{1}$ and $J_{2}$
%The number of points of generalized Steinberg variety of $H$ corresponding to $J_{1}$ and $J_{2}$ is given by
\[
|St_{H}(J_1,J_2)|=\frac{|H|}{|P_{1}|}\sum_{h\in H(\mathbb{F}_{q})/P_2(\mathbb{F}_{q})}|\mathcal{N}(\text{Lie}(P_1\cap h\cdot P_2))|
=\frac{|H|}{|P_{1}|}|P_1|\sum_{h\in P_{1}(\mathbb{F}_{q})\backslash H(\mathbb{F}_{q})/P_2(\mathbb{F}_{q})}\frac{|\mathcal{N}(\text{Lie}(P_1\cap h\cdot P_2))|}{|P_{1}\cap h\cdot P_2|},
\]
where we have also used $|H/P_{1}|=|H|/|P_{1}|$ for the first equality. Now using $P_{1}(\mathbb{F}_{q})\backslash H(\mathbb{F}_{q})/P_2(\mathbb{F}_{q})\cong W_{1}\backslash W_{H}/W_2$ (see \cite[Theorem 65.21]{CR} , \cite[Theorem 21.91]{Mil1} and use the well-known fact that $H(\mathbb{F}_{q})$ is a finite group with a $BN$-pair \cite[Proposition 21.75]{Mil1} for $B=B_{H}(\mathbb{F}_{q}), N=N_{T_{H}}(\mathbb{F}_{q})$, where $N_{T_{H}}$ is the normalizer of $T_H$ in $H$), we have
\[
|St_{H}(J_1,J_2)|
=
|H|\sum_{w\in W_{1}\backslash W_{H}/W_{2}}\frac{|\mathcal{N}(\text{Lie}(P_1\cap w\cdot P_2))|}{|P_{1}\cap w\cdot P_2|}=|H|\sum_{w\in W_{1}\backslash W_{H}/W_2}\frac{q^{\dim(P_1\cap w\cdot P_2)-\rk(P_1\cap w\cdot P_2)}}{|P_{1}\cap w\cdot P_2|},
\]
where the second equality follows from $|\mathcal{N}(\text{Lie}(P_{1}\cap w\cdot P_{2}))|=q^{\dim(P_{1}\cap w\cdot P_{2})-\rk(P_{1}\cap w\cdot P_{2})}$.

Next, we have the following decomposition (the statement is easily reduced to $\overline{\mathbb{F}}_{q}$ in which case it is given by \cite[Proposition 2.15]{DM}):
\begin{equation}\label{decomposition of intersection}
P_{1}\cap w\cdot P_{2}=(L_{1}\cap w\cdot L_{2})(L_{1}\cap w\cdot U_{2})(U_{1}\cap w\cdot L_{2})(U_{1}\cap w\cdot U_{2})
\end{equation}
which is a direct product of varieties over $\mathbb{F}_{q}$. Now using the fact that for any unipotent algebraic group $U$ over $\mathbb{F}_{q}$, we have $|U|=q^{\text{dim}(U)}$ (see \cite[Corollary 14.2.7]{Spr}), we obtain
\[
|St_{H}(J_1,J_2)|=|H|\sum_{w\in W_{1}\backslash W_{H}/W_2}\frac{q^{\dim(L_1\cap w\cdot L_2)-\rk(L_1\cap w\cdot L_2)}}{|L_{1}\cap w\cdot L_2|}=|H|\sum_{w\in W_{1}\backslash W_{H}/W_2}\frac{|\mathcal{N}(\text{Lie}(L_{1}\cap w\cdot L_{2}))|}{|L_{1}\cap w\cdot L_2|}.
\]
where we use $|\mathcal{N}(\text{Lie}(L_{1}\cap w\cdot L_{2}))|=q^{\dim(L_{1}\cap w\cdot L_{2})-\rk(L_{1}\cap w\cdot L_{2})}$ for the second equality. 
Recall $D^{H}_{J_{1},J_{2}}$ from Section \ref{coproduct} and let $w\in D^{H}_{J_{1},J_{2}}$. In this case, we also have the following decomposition (the statement is easily reduced to $\overline{\mathbb{F}}_{q}$ in which it is given by \cite[Theorem 2.8.7]{Car}):
\begin{equation}\label{another decomposition of intersection}
    P_{1}\cap w\cdot P_{2}=(L_{J_{1}\cap w\cdot J_{2}})(L_{1}\cap w\cdot U_{2})(U_{1}\cap w\cdot L_{2})(U_{1}\cap w\cdot U_{2})
\end{equation}
By \eqref{decomposition of intersection}, \eqref{another decomposition of intersection} and the fact that $L_{J_{1}\cap w\cdot J_2}\subset L_{1}\cap w\cdot L_{2}$, we get $L_{J_{1}\cap w\cdot J_2}=L_{1}\cap w\cdot L_{2}$, which gives
\begin{equation}\label{steinberglevi}
    |St_{H}(J_1,J_2)|=|H|\sum_{w\in D^{H}_{J_{1},J_{2}}}\frac{|\mathcal{N}(\text{Lie}(L_{J_{1}\cap w\cdot J_{2}}))|}{|L_{J_{1}\cap w\cdot J_2}|}.
\end{equation}
Recalling that $\Delta_{H}$ is given by
\[
\Delta_{H}(f)(J_1,J_2)=\sum_{w\in D^{H}_{J_{1},J_{2}}}f(J_{1}\cap w\cdot J_{2}),
\]
we get from $\eqref{levi1}$ and $\eqref{steinberglevi}$ that
\[
\Delta_{H}([Sp_{H}])=[St_{H}].
\]
This finishes the proof of part ($ii$) of Theorem \ref{steinberg}.
\qed
\subsection{More on coproduct} 
In this section, we would like to prove a few properties of $\Delta_H$ that are of independent interest and will be used later in Section \ref{case of gln} in the case of $GL_n$. First we need some definitions.
\begin{definition}
Let $J_{1}, J_{2} \subset \Pi_H$, we say $J_{1}$ and $J_{2}$ are \emph{associates} whenever $\Phi_{J_{2}}=w\cdot \Phi_{J_{1}}$ for some $w\in W_H$. This gives an equivalence relation on $\mathcal{P}(\Pi_{H})$, which we denote by $\sim_H$. Let $f\in\mathbb{Z}[\mathcal{P}(\Pi_{H})]$, we say $f$ is \emph{associate invariant} if $f(J_{1})=f(J_{2})$ whenever $J_{1}$ and $J_{2}$ are associates. We say that a function of two variables is \emph{associate invariant} if it is associate invariant in each variable.
\end{definition}
Let $\mathcal{O}$ be an equivalence class of $\sim_{H}$.
%, let us fix a representative $J$. 
Let $\delta_{\mathcal{O}}\in\mathbb{Z}[\mathcal{P}(\Pi_{H})]$ be the function on $\mathcal{P}(\Pi_{H})$ that takes the value $1$ on $J$ if $J\in\mathcal{O}$ and $0$ otherwise. Let us fix a representative $J_{\mathcal{O}}$ in each equivalence class $\mathcal{O}$. 
The following lemma states that $\Delta_{H}$ preserves associate invariant functions.
\begin{lemma}\label{associateinvariantdelta}
%Keep notations as above. Then
We have
\begin{equation}\label{lhsrhs}
\Delta_{H}(\delta_{\mathcal{O}})=\sum_{(\mathcal{O}_{1},\mathcal{O}_{2})\in\big(\mathcal{P}(\Pi_{H})/\sim\big)\times\big(\mathcal{P}(\Pi_{H})/\sim\big)}n_{\mathcal{O}}^{\mathcal{O}_{1},\mathcal{O}_{2}}\delta_{\mathcal{O}_{1}}\otimes \delta_{\mathcal{O}_{2}},
\end{equation}
where
\[
n_{\mathcal{O}}^{\mathcal{O}_{1},\mathcal{O}_{2}}
=
\big|\{w\in W_{J_{\mathcal{O}_{1}}}\backslash W_{H}/W_{J_{\mathcal{O}_{2}}}:\Phi_{J_{\mathcal{O}_{1}}}\cap w\cdot \Phi_{J_{\mathcal{O}_{2}}}=
w'\cdot\Phi_{J_{\mathcal{O}}}\text{ for some }w'\in W_H\}\big|.
\]
In particular, $\Delta_H$ preserves the subspace of associate invariant functions.
\end{lemma}
\begin{proof}
First we rewrite the coproduct $\Delta_{H}$ for associate invariant functions. Set $\mathcal{R}(J):=\Phi_{J}$, so that $\mathcal{R}$ is a bijection from $P(\Pi_{H})$ onto the set of root systems of all Levi subgroups of $H$ containing $T_H$. Let $f\in\mathbb{Z}[\mathcal{P}(\Pi_{H})]$ be associate invariant. Then for $(J_{1},J_{2})\in\mathcal{P}(\Pi_{H})\times\mathcal{P}(\Pi_{H})$,
\begin{equation}\label{different formulation of coproduct}
    \Delta_{H} (f)(J_1,J_2)=\sum_{w\in W_{J_1}\setminus W_{H}/W_{J_2}}f\big(\mathcal{R}^{-1}(\Phi_{J_{1}}\cap w\cdot\Phi_{J_{2}})\big).
\end{equation}
We note that in this reformulation of $\Delta_{H}$ for associate invariant functions the summands does not depend on a particular choice of the element of a double coset.
For any $J_{1}, J_{2}\in\mathcal{P}(\Pi_{H})$, $\Delta_{H}(\delta_{\mathcal{O}})$ evaluated at $(J_{1},J_{2})$ is equal to
\[
\sum_{w\in W_{J_{1}}\backslash W_{H}/W_{J_{2}}}\delta_{\mathcal{O}}\big(\mathcal{R}^{-1}(\Phi_{J_{1}}\cap w\cdot \Phi_{J_{2}})\big),
\]
which in turn is equal to
\[
\big|\{w\in W_{J_{1}}\backslash W_{H}/W_{J_{2}}:\Phi_{J_{1}}\cap w\cdot \Phi_{J_{2}}=w'\cdot\Phi_{J_{\mathcal{O}}}\text{ for some }w'\in W_H\}\big|.
\]
On the other hand, RHS of \eqref{lhsrhs} evaluated at $(J_{1},J_{2})$ is equal to $n^{\mathcal{O}_{1},\mathcal{O}_{2}}_{\mathcal{O}}$, where $\mathcal{O}_{1}$ (resp. $\mathcal{O}_2$) is the equivalence class of $J_{1}$ (resp. $J_{2}$). There exists $w_{1}, w_{2} \in W_{H}$ such that $\Phi_{J_{1}}=w_{1}\cdot \Phi_{J_{\mathcal{O}_{1}}}$, $\Phi_{J_{2}}=w_{2}\cdot \Phi_{J_{\mathcal{O}_{2}}}$ and so, $W_{J_1}=w_{1}W_{J_{\mathcal{O}_{1}}}w_{1}^{-1}$ and $W_{J_2}=w_{2}W_{J_{\mathcal{O}_{2}}}w_{2}^{-1}$.
Now the lemma follows from the bijection
\[
W_{J_{\mathcal{O}_{1}}}\backslash W_{H}/W_{J_{\mathcal{O}_{2}}}\rightarrow W_{J_{1}}\backslash W_{H}/W_{J_{2}}, \quad W_{J_{\mathcal{O}_{1}}}wW_{J_{\mathcal{O}_{2}}}\mapsto W_{J_{1}}(w_{1}ww_{2}^{-1})W_{J_{2}}.
\]
This finishes the proof of Lemma \ref{associateinvariantdelta}.
\end{proof}
\begin{remark}\label{cocommutative}
  The proof of Lemma \ref{associateinvariantdelta} suggests that \eqref{different formulation of coproduct} may be a better definition for $\Delta_{H}$ as it does not use \cite[Proposition 2.7.3]{Car}. In fact, it may be even better to view $f$ as a function on the set of root systems of the Levi subgroups. Moreover, using this formulation it is easy to see that $\Delta_{H}$ is co-commutative for associate invariant functions.
\end{remark}
%As a consequence of Lemma \ref{associateinvariantdelta},
We have the following corollary.
\begin{corollary}\label{sp and st are associate invariant}
Let $[Sp_{H}]$ and $[St_{H}]$ be as in Section \ref{coproduct}. Then $[Sp_{H}]$ and $[St_{H}]$ are associate invariant functions.
\end{corollary}
\begin{proof}
Let $J, J'\in\Pi_{H}$ be such that $J\sim_{H}J'$. Then we have $L_{J}\simeq L_{J'}$ and as a consequence of \eqref{levi1}, it follows that $[Sp_{H}]$ is associate invariant. Now Lemma \ref{associateinvariantdelta} together with Theorem \ref{steinberg}$(ii)$ imply that $[St_{H}]$ is associate invariant in each variable.
\end{proof}
Assume that $H=H_1\times\ldots\times H_n$. For $k=1,\ldots, n$, let $\Pi_k$ be the set of simple roots of $H_k$ with respect to some maximal torus and a Borel subgroup containing it. We can identify $\Pi_{H}$ with the disjoint union $\bigsqcup_k\Pi_k$. Thus, $\mathcal{P}(\Pi_{H})=\prod_k\mathcal{P}(\Pi_k)$ and $\mathbb{Z}[\mathcal{P}(\Pi_{H})]=\bigotimes_k\mathbb{Z}[\mathcal{P}(\Pi_k)]$. Under this isomorphism, the following lemma follows from the definitions.
\begin{lemma}\label{lm1}
%Keep notations as above. Then
We have
\[
    [St_H]=[St_{H_1}]\otimes\ldots\otimes[St_{H_n}].
\]
\end{lemma}

\section{Bialynicki--Birula decomposition}
In this section we recall the Bialynicki--Birula decomposition. We will use these facts in the next section to give a proof of Theorem \ref{trip}.
\begin{definition}\label{affine fibration}
Let $X$ and $Z$ be two schemes. A morphism $\phi : X \rightarrow Z$ is called an \emph{affine fibration} of relative dimension $d$ if for every $z \in Z$, there is a Zariski open neighborhood $U$ of $z$ such that $X_{U} \cong U \times \mathbb{A}^{d}$ and this isomorphism identifies $\phi_{|_{U}} : X_{U} \rightarrow Z$ with the projection on the first factor.

A morphism $\phi : X \rightarrow Z$ is called a \emph{trivial affine fibration} of relative dimension $d$ if $X \cong Z \times \mathbb{A}^{d}$ and this isomorphism identifies $\phi : X \rightarrow Z$ with the projection on the first factor.
\end{definition}
We use the following result (see \cite[Theorem 3.2]{Bro}), known as the Bialynicki--Birula decomposition which is key to our calculation:

\begin{fact}\label{bb1} 
(Bialynicki--Birula, Hesselink, Iversen). 
%Let $K$ be any field. 
Let $X$ be a smooth,
projective scheme over $k$ equipped with a $\mathbb{G}_{m}$-action. Then the following holds:
\begin{enumerate}[(i)]
    \item The fixed point locus $X^{\mathbb{G}_{m}}$ is a closed subscheme of $X$ and is smooth over $k$.
    \item There exists a numbering $X^{\mathbb{G}_{m}} = \bigsqcup_{i=1}^{n} Z_{i}$ of the connected components
    of $X^{\mathbb{G}_{m}}$, and a filtration of $X$ by closed subschemes:
    \[
      X=X_{n}\supset X_{n-1}\supset...\supset X_{0}\supset X_{-1}=\emptyset
    \]
    and affine fibrations $\phi_{i}:X_{i}-X_{i-1}\rightarrow Z_{i}$.
    \item The relative dimension of $\phi_{i}$ is the dimension of the positive eigenspace of the
    $\mathbb{G}_{m}$-action on the tangent space of $X$ at
    an arbitrary closed point $z \in Z_{i}$ and $\dim(Z_{i})=\dim(T_{z,X}^{\mathbb{G}_{m}})$ .
\end{enumerate}
\end{fact}
In particular, we obtain a stratification of $X$ by locally closed subsets $X_{i}^{+}:=X_{i}-X_{i-1}$. 
\begin{definition}
%Let $K$ be a field. 
Let $Y$ be a separated scheme over $k$. Let $\phi:\mathbb{A}^{1}\backslash\{0\}\rightarrow Y$ be a morphism. If $\phi$ extends to a morphism $\widetilde{\phi}:\mathbb{A}^{1}\rightarrow Y$, we say that $\lim_{t\to 0}\phi(t)$ exists and we set it equal to $\widetilde{\phi}(0)$. Since $Y$ is separated over $k$ and $\mathbb{A}^{1}$ is reduced, the extension $\tilde{\phi}$ is unique. Note that if, moreover, $Y$ is proper over $k$ then an extension of $\phi$ always exists. 
\end{definition}

\begin{remark} \label{remark 1}
(see \cite[Section 3]{Bro}) The Bialynicki--Birula decomposition is explicit in the sense that the locally closed subsets $X_{i}^{+}$ is the set of all points $x \in X$ such that
$\lim_{t\to 0}t\cdot x \in Z_{i}$ where $(t, x) \mapsto t\cdot x$ is the $\mathbb{G}_{m}$-action. Moreover, the map $\phi_{i} : X_{i}^{+} \rightarrow Z_{i}$ is then given by $x \mapsto \lim_{t\to 0} t\cdot x$. 
\end{remark}
Let $k$ be a field. Let $S$ be a smooth separated %noetherian
scheme
over $k$ equipped with a $\mathbb{G}_{m}$-action. By \cite[Proposition A.8.10]{CGP}, $S^{\mathbb{G}_{m}}$ is smooth over $k$. By a \emph{smooth equivariant compactification} of $S$, we will mean a scheme $\overline{S}$ that is smooth and projective over $k$, $S$ is an open and dense subscheme of $\overline{S}$ and $\overline{S}$ is equipped with a $\mathbb{G}_{m}$-action that extends the $\mathbb{G}_{m}$-action on $S$. The following proposition is a consequence of Fact \ref{bb1}.
\begin{proposition}\label{proposition1}
Assume that there is a smooth equivariant compactification $\overline{S}$ of $S$.
Let $S^{\fin}$ be the subset of $S$ consisting of points $x$ in $S$ for which $\lim_{t\to 0}t\cdot x$ exists in $S$.
Then $S^{\fin}$ is a constructible subset of $S$ and there exists a stratification of $S^{\fin}$ by locally closed subsets as:
\[
S^{\fin}=\displaystyle\bigsqcup_{\alpha\in I}S^{+}_{\alpha}
\]
and a decomposition of $S^{\gm}$ as:
\[
S^{\gm}=\bigsqcup_{\alpha\in I}S_{\alpha}^{\gm},
\]
where $S_{\alpha}$ are the connected components of $S^{\gm}$, $\alpha\in I$. 
Moreover, there are affine fibrations $lim_{\alpha}:S_{\alpha}^{+}\rightarrow S_{\alpha}^{\gm}$ given by the limit map as $t\to 0$.
\end{proposition}
In the case of an equivariant vector bundle over a smooth projective scheme equipped with a $\gm$-action, we can say a bit more about the strata in Proposition \ref{proposition1}. 
Let $k$ be a field and let $X$ be a smooth projective 
scheme over $k$ equipped with a $\mathbb{G}_{m}$-action. Let $\pi:E\rightarrow X$ be an equivariant vector bundle over $X$.
Compactify $E$ by considering the projectivization $
\mathbb{P}\big(E\oplus (X\times\mathbb{A}^{1})\big)=:\overline{E}$. 
We extend the given $\mathbb{G}_{m}$-action on $E$
to a $\gm$-action on $\overline{E}$ by letting
$\gm$ act trivially on $\mathbb A^{1}$ and via the given $\gm$-action on $X$. Since  projectivization of a vector bundle over a smooth scheme is smooth, $\overline{E}$
is smooth. Thus $\overline{E}$ is a smooth equivariant compactification of $E$.

Now let us consider the Bialynicki--Birula decomposition of $X$. By Fact \ref{bb1}, $X$ has a stratification by locally closed subsets as:
\[
X=\bigsqcup_{\alpha\in I}X_{\alpha}^{+}
\]
and a decomposition of $X^{\gm}$ as:
\[
X^{\gm}=\bigsqcup_{\alpha\in I}X_{\alpha}^{\gm}
\]
where $X_{\alpha}^{\gm}$ are the connected components of $X^{\gm}$, $\alpha\in I$.

Since $\gm$ acts trivially on $X_{\alpha}^{\gm}$ and $\pi$ is $\gm$-equivariant, $\gm$ acts on the vector bundle $\pi^{-1}(X_{\alpha}^{\gm})\rightarrow X_{\alpha}^{\gm}$
fibrewise. Therefore, 
$\pi^{-1}(X_{\alpha}^{\gm})$ decomposes according to the characters of $\gm$, 
\[
\pi^{-1}(X_{\alpha}^{\gm})=\oplus_{n\in \mathbb{Z}}V_{\alpha,n},
\]
where $V_{\alpha,n}$ is the subbundle of $\pi^{-1}(X_{\alpha}^{\gm})$ on which $t\in\gm$ acts 
via multiplication by $t^n$. We have the following proposition.
\begin{proposition}\label{stratification of bundle}
%Keep notations as above and as in Proposition \ref{proposition1}.
%Then 
We have $E^{\fin}$ is a constructible subset of $E$ and there exists a stratification of $E^{\fin}$ by locally closed subsets as:
\[
E^{\fin}=\displaystyle\bigsqcup_{\alpha\in I}E^{+}_{\alpha}
\]
and a decomposition of $E^{\gm}$ as:
\[
E^{\gm}=\bigsqcup_{\alpha\in I}V_{\alpha,0},
\]
where $V_{\alpha,0}$ are the connected components of $E^{\gm}$, $\alpha\in I$ %where $I$ is a finite indexing set.
and there are affine fibrations $lim_{\alpha}:E_{\alpha}^{+}\rightarrow V_{\alpha,0}$ given by the limit map as $t\to 0$.
\end{proposition}
\begin{proof}
Notice that we have
$
E^{\gm}=\bigsqcup_{\alpha\in I}V_{\alpha,0}.
$
Since $V_{\alpha,0}=\big(\pi^{-1}(X_{\alpha}^{\gm})\big)^{\gm}$, $V_{\alpha,0}$ is closed, $\alpha\in I$. Moreover, since %$J$ is finite and 
$V_{\alpha,0}$, $\alpha\in I$ are connected, we get that $V_{\alpha,0}$, $\alpha\in I$ are the connected components of $E^{\gm}$. 
It remains to use Proposition \ref{proposition1}.
\end{proof}
\section{Counting triples}\label{formula}
In this section we prove our main result Theorem \ref{trip}, i.e., we construct a stratification of our moduli space of triples into pieces that can be described in terms of affine fibrations over the generalized Steinberg varieties. %This section will be devoted to the proof of Theorem \ref{trip}.

Let $G,T,B,\Pi,W,\Phi$ be as in Section \ref{parabolic-levi} and let $\mu,J_{0},J_{\infty}$ be as in the statement of Theorem \ref{trip}. Let $\mathfrak{g}:=\lie(G)$ be the Lie algebra of $G$. 
Since $\mu$, $J_{0}$ and $J_{\infty}$ are fixed in the statement of Theorem \ref{trip}, we will denote $\mathcal{T}rip_{\mu}(J_{0},J_{\infty})$ by $\mathcal{T}rip$ in the proof of Theorem \ref{trip}.
\subsection{Strategy of the proof}\label{strategy of the proof}
In this section, we outline the strategy of the proof of Theorem \ref{trip}.
Let $\mathbb{G}_{m}$ act on $\mathfrak{g}$ via $\mu$, so $t\in\gm$ acts trivially on $\mathfrak{h}$ and via multiplication by $t^{\langle\alpha,\mu\rangle}$ on the root spaces $\mathfrak{g}_{\alpha}$. 
Let $\mathfrak{g}^{0}:=\mathfrak{h}\oplus_{\langle\alpha,\mu\rangle=0}\mathfrak{g}_{\alpha},\mathfrak{g}^{+}:=\oplus_{\langle\alpha,\mu\rangle>0}\mathfrak{g}_{\alpha}$ and $\mathfrak{g}^{-}:=\oplus_{\langle\alpha,\mu\rangle<0}\mathfrak{g}_{\alpha}$. Then we get the following $\mathbb{G}_{m}$-stable decomposition of $\mathfrak{g}$:
\begin{equation}\label{decomposition of g}
\mathfrak{g}=\mathfrak{g}^{0}\oplus\mathfrak{g}^{+}\oplus\mathfrak{g}^{-}.
\end{equation}
Note that we have $\mathfrak{g}^{0}=\lie(L_{\mu})$.

For $J\subset\Pi$, define $\mathcal{B}_{J}$ to be the scheme of pairs $(\mathfrak{p},v)$ such that $\mathfrak{p}\in G/P_{J},v\in\mathfrak{p}$,
where we identify $G/P_{J}$ with the scheme of parabolic subalgebras of $G$ that are 
in the $G$-orbit of $\mathfrak{p}_{J}:=\lie(P_{J})$ under the adjoint action. Note that $\mathcal{B}_{J}$ is vector bundle over $G/P_{J}$, in fact, it is a vector subbundle of the trivial vector bundle $G/P_{J}\times\mathfrak{g}$ over $G/P_{J}$. As vector bundles over smooth schemes are smooth, we get that 
$\mathcal{B}_{J}$ is smooth. Note that $G$ acts in a natural way on $G/P_{J}\times\mathfrak{g}$ preserving $\mathcal{B}_{J}$. Pulling back this action along $\mu:\gm\rightarrow T\rightarrow G$, we get an action
\begin{equation}\label{action on B_J}
    \gm\curvearrowright\mathcal{B}_{J}.
\end{equation} 
We introduce the following object for our proof of Theorem \ref{trip}.
\begin{definition}\label{definition of quad}
Let $\cQuad$ be the closed subscheme of $\mathcal{B}_{J_{0}}\times\mathcal{B}_{J_{\infty}}$ consisting of quadruples $(\mathfrak{p}_{0},v_0,\mathfrak{p}_{\infty},v_{\infty})$ such that $v_{0}$ and $v_{\infty}$ are nilpotent and with respect to the decomposition \eqref{decomposition of g}, the $\mathfrak{g}^{-}$-components of $v_{0}$ and $v_{\infty}$ are zero and their $\mathfrak{g}^{0}$-components are equal.  
%Let $\cQuad$ be the closed subscheme of $\mathcal{B}_{J_{0}}\times\mathcal{B}_{J_{\infty}}$ consisting of quadruples $(P_{0},v_0,P_{\infty},v_{\infty})$ such that $v_{0}$ and $v_{\infty}$ are nilpotent, $\mathfrak{g}^{-}$-components \eqref{decomposition of g} of $v_{0}$ and $v_{\infty}$ are zero and their $\mathfrak{g}^{0}$-components \eqref{decomposition of g} are equal.  
\end{definition}
Note that $\cQuad$ depends on $\mu$, $J_{0}$ and $J_{\infty}$.
\begin{remark}\label{equivalent definition of Quad}
The requirement of $v_{0}$ and $v_{\infty}$ being nilpotent in the definition of $\cQuad$ is equivalent to the requirement of the $\mathfrak{g}^{0}$-components of $v_{0}$ and $v_{\infty}$ being nilpotent.
\end{remark}
Recall $\mathcal{B}_{J}^{\fin}$ from Proposition \ref{proposition1}. Since $\mathcal{B}_{J}$ is an equivariant vector bundle over $G/P_{J}$, we stratify $\mathcal{B}_{J}^{\fin}$ by applying Proposition \ref{stratification of bundle} on $\mathcal{B}_{J}$.
We obtain the required stratification of $\mathcal{T}rip$ in the following manner: trivialize the fibers of the line bundle $\mathcal{O}(1)$ at $0$ and $\infty$ to identify ad$(\cE_{\mu})_{0}$ and ad$(\cE_{\mu})_{\infty}$ with $\mathfrak{g}$, now evaluating the nilpotent sections at $0$ and $\infty$ gives us a $\gm$-equivariant morphism $\mathcal{T}rip\rightarrow\mathcal{B}_{J_{0}}\times\mathcal{B}_{J_{\infty}}$ with $\gm$ acting diagonally on $\mathcal{B}_{J_{0}}\times\mathcal{B}_{J_{\infty}}$. We will see in Lemma \ref{quad} that this evaluation morphism is a trivial affine fibration onto its image, which is equal to $\cQuad$. Thus it is enough to stratify $\cQuad$.
We show that for points in $\cQuad$, the limit exists in $\mathcal{B}_{J_{0}}\times\mathcal{B}_{J_{\infty}}$ as $t\to 0$
(see Lemma \ref{lemma}), so $\cQuad\subset\mathcal{B}_{J_{0}}^{\fin}\times\mathcal{B}_{J_{\infty}}^{\fin}$.
We will see in Lemma \ref{quadquad} that intersecting the strata of $\mathcal{B}_{J_{0}}^{\fin}\times\mathcal{B}_{J_{\infty}}^{\fin}$ with $\cQuad$, we obtain a
stratification of $\cQuad$.
\subsection{Reduction to $\cQuad$}
Now we consider %properties of 
evaluations of the nilpotent sections of ad$(\cE_{\mu})$ at $0$ and $\infty$ and then use them to reduce Theorem \ref{trip} to finding a stratification of $\cQuad$.
Recall that as $\mathcal{O}(1)^{\times}$ is a $\mathbb{G}_{m}$-bundle over $\mathbb{P}^1$, $\mathbb{G}_{m}$ acts on ad$(\cE_{\mu})=\mathcal{O}(1)^{\times}\times^{\mathbb{G}_{m}}\mathfrak{g}$ (Section \ref{basic_principal bundles}) and this gives an action:
\begin{equation}\label{action on sections}
    \gm\curvearrowright H^{0}(\mathbb{P}^{1},\text{ad}(\cE_{\mu})).
\end{equation}
First, we describe sections of the adjoint bundle ad$(\cE_{\mu})$ over $\mathbb{P}^1$. Since $\text{ad}(\cE_{\mu})=\mathcal{O}(1)^{\times}\times^{\gm}\mathfrak{g}$, the $\mathbb{G}_{m}$-stable decomposition \eqref{decomposition of g} of $\mathfrak{g}$ gives a $\mathbb{G}_{m}$-stable decomposition of $\text{ad}(\cE_{\mu})$ as 
$\text{ad}(\cE_{\mu})=\text{ad}(\cE_{\mu})^{0}\oplus\text{ad}(\cE_{\mu})^{+}\oplus\text{ad}(\cE_{\mu})^{-}$, 
where 
$\text{ad}(\cE_{\mu})^{0}:=\mathcal{O}(1)^{\times}\times^{\mathbb{G}_{m}}\mathfrak{g}^{0}, \text{ad}(\cE_{\mu})^{+}:=\mathcal{O}(1)^{\times}\times^{\mathbb{G}_{m}}\mathfrak{g}^{+}$ and $\text{ad}(\cE_{\mu})^{-}:=\mathcal{O}(1)^{\times}\times^{\mathbb{G}_{m}}\mathfrak{g}^{-}$.
Since $\text{ad}(\cE_{\mu})^{-}$ is a direct sum of the line bundles $\mathcal{O}(m)$, $m<0$ and $H^{0}(\mathbb{P}^{1},\mathcal{O}(m))=0$ for $m<0$, we get
\[
H^{0}(\mathbb{P}^{1},\text{ad}(\cE_{\mu}))=
\mathfrak{g}^{0}\oplus H^{0}\big(\mathbb{P}^{1},\text{ad}(\cE_{\mu})^{+}\big).
\]
Thus
\begin{equation}\label{global sections}
H^{0}(\mathbb{P}^{1},\text{ad}(\cE_{\mu}))
=
\mathfrak{g}^{0}\oplus \Big(\oplus_{\alpha:\langle\alpha,\mu\rangle>0}H^{0}\big(\mathbb{P}^{1},\mathcal{O}(\langle\alpha,\mu\rangle)\big)\Big).
\end{equation}
For $x=0,\infty$, $\text{ad}(\cE_{\mu})_{x}$ has a structure of a Lie algebra and for $\Psi\in H^{0}(\mathbb{P}^{1},\text{ad}(\cE_{\mu}))$, denote the value of $\Psi$ at $x$ by $\Psi_x$, which is an element of $\text{ad}(\cE_{\mu})_{x}$.

We get the following $\gm$-stable decomposition of $\text{ad}(\cE_{\mu})_{x}$:
\[
\text{ad}(\cE_{\mu})_{x}=\text{ad}(\cE_{\mu})_{x}^{0}\oplus\text{ad}(\cE_{\mu})_{x}^{+}\oplus\text{ad}(\cE_{\mu})_{x}^{-},\quad x=0,\infty.
\]

\begin{remark}\label{fiber isomorphism}
By trivializing the fibers of the $\gm$-bundle $\mathcal{O}(1)^{\times}$ at $0$ and $\infty$, we identify $(\cE_{\mu})_{x}/P_{J_{x}}$ with $G/P_{J_{x}}$ and we get
a $\gm$-equivariant isomorphism (which is fixed from now on) $\text{ad}(\cE_{\mu})_{x}\cong\mathfrak{g}$, which maps $\text{ad}(\cE_{\mu})_{x}^{0}$ isomorphically onto $\mathfrak{g}^{0}$, $x=0,\infty$. We note that the isomorphism $\text{ad}(\cE_{\mu})_{x}^{0}\cong\mathfrak{g}^{0}$ is independent of the trivialization. From now on, we will use the isomorphism $\text{ad}(\cE_{\mu})_{x}\cong\mathfrak{g}$ to identify elements of $\text{ad}(\cE_{\mu})_{x}$ with those of $\mathfrak{g}$, $x=0,\infty$.
\end{remark}
The $\mathbb G_{m}$-action \eqref{action on B_J} on $\mathcal{B}_{J_{x}}$, $x=0,\infty$ gives a $\gm$-action on $\mathcal{B}_{J_{0}}\times\mathcal{B}_{J_{\infty}}$ by $\gm$ acting diagonally. Since the decomposition \eqref{decomposition of g} is $\gm$-stable, we get an action
\begin{equation}\label{action on quad}
\gm\curvearrowright\mathcal{Q}uad.
\end{equation}
Using Remark \ref{fiber isomorphism}, we consider the evaluation morphism at $0$ and $\infty$ as taking values in $\mathcal{B}_{J_0}\times\mathcal{B}_{J_\infty}$:
\[
    ev^{0,\infty}:\mathcal{T}rip\rightarrow\mathcal{B}_{J_0}\times\mathcal{B}_{J_\infty},\quad
(\mathfrak{p}_{0},\mathfrak{p}_{\infty},\Psi)\mapsto (\mathfrak{p}_{0},\Psi_0,\mathfrak{p}_{\infty},\Psi_\infty).
\]
Consider the evaluation map at $0$ and $\infty$,
\[
    eval:H^{0}(\mathbb{P}^{1},\text{ad}(\cE_{\mu}))\rightarrow\mathfrak{g}\oplus\mathfrak{g},\quad\Psi\mapsto(\Psi_{0},\Psi_{\infty}).
\]
Notice that for $\Psi\in\mathfrak{g}^{0}$, $eval(\Psi)=(\Psi,\Psi)$. Since $\Psi\in H^{0}(\mathbb{P}^{1},\text{ad}(\cE_{\mu}))$ is nilpotent if and only if the $\mathfrak{g}^{0}$-component of $\Psi$ is nilpotent and the evaluation map $H^{0}(\mathbb{P}^{1},\mathcal{O}(m))\rightarrow\mathbb{A}^{1}\times\mathbb{A}^{1}$, $\phi\mapsto(\phi_{0},\phi_{\infty})$ is surjective for $m>0$, the image of $ev^{0,\infty}$ is equal to $\cQuad$.

The next lemma relates $\mathcal{T}rip$ and $\cQuad$ via evaluation at $0$ and $\infty$.

\begin{lemma}\label{quad}
The evaluation morphism
\[
ev^{0,\infty} : \mathcal{T}rip\rightarrow \cQuad
\]
is $\gm$-equivariant and a trivial affine fibration of relative dimension $\sum_{ \langle\alpha,\mu\rangle>0}\Big(\langle\alpha,\mu\rangle-1\Big)$. Moreover, $ev^{0,\infty}$ gives the following commutative triangle:
\[
\begin{tikzcd}[row sep=4.5em, column sep=2.5em]
    \mathcal{T}rip \arrow[rr, "\sim"] \arrow[swap, dr, "ev^{0,\infty}"] & & \cQuad\times W \arrow[dl, "pr_1"] \\
    & \cQuad \\[-4.6em]
\end{tikzcd}
\]
where $W$ is a $\gm$-representation with $\gm$ acting by positive weights and all morphisms in the above triangle are $\gm$-equivariant. In particular, $ev^{0,\infty} : \mathcal{T}rip\rightarrow \cQuad$ induces an isomorphism
\begin{equation}\label{isomorphism}
      ev^{0,\infty} : \mathcal{T}rip^{\gm}\xrightarrow{\sim} \cQuad^{\gm}.
\end{equation}
\end{lemma}
\begin{proof}
Put
\[
\mathfrak{g}_{0,\infty}:=\{(v_{0},v_{\infty})\in\mathfrak{g}\oplus\mathfrak{g}:\mathfrak{g}^{0}\text{-components of }
v_{x}\text{ are equal},
 \text{ }\mathfrak{g}^{-}\text{-components of }v_{x} \text{ are } 0, \text{ }x=0,\infty\}.
\]
Since the image of $eval$ lies inside $\mathfrak{g}_{0,\infty}$,
we will consider $eval$ with codomain $\mathfrak{g}_{0,\infty}$,
\[
eval:H^{0}(\mathbb{P}^{1},\text{ad}(\cE_{\mu}))
\rightarrow \mathfrak{g}_{0,\infty}.
\]
Since the evaluation map $H^{0}(\mathbb{P}^{1},\mathcal{O}(m))\rightarrow\mathbb{A}^{1}\times\mathbb{A}^{1}$, $\phi\mapsto(\phi_{0},\phi_{\infty})$ is surjective for $m>0$ and $eval(v)=(v,v)$ for $v\in\mathfrak{g}^0$, the morphism $eval$ is surjective. 

Let $W:=\ker(eval)$. Notice that $W$ is a $\gm$-representation with $\gm$ acting by positive weights. Since $\gm$ is reductive, we get a $\gm$-equivariant isomorphism:
\[
H^{0}(\mathbb{P}^{1},\text{ad}(\cE_{\mu}))\cong W\times\mathfrak{g}_{0,\infty}.
\]
Denote the nilpotent elements of $H^{0}(\mathbb{P}^{1},\text{ad}(\cE_\mu))$ by $H^{0}(\mathbb{P}^{1},\text{ad}(\cE_\mu))^{nil}$. 
Let $\mathfrak{g}^{nil}_{0,\infty}$ denote the set of elements $(v_{0},v_{\infty})\in\mathfrak{g}_{0,\infty}$ such that the $\mathfrak{g}^{0}$-components of $v_{0}$ and $v_{\infty}$ are nilpotent.
Since $\Psi\in H^{0}(\mathbb{P}^{1},\text{ad}(\cE_\mu))$ is nilpotent if and only if the $\mathfrak{g}^{0}$-component of $\Psi$ is nilpotent, 
we get a $\gm$-equivariant isomorphism
\[
     H^{0}(\mathbb{P}^{1},\text{ad}(\cE_\mu))^{nil}\cong W\times\mathfrak{g}^{nil}_{0,\infty}.
\]
Since
$ev^{0,\infty} : \mathcal{T}rip\rightarrow \cQuad$ is the pullback of $H^{0}(\mathbb{P}^{1},\text{ad}(\cE_\mu))^{nil}\xrightarrow{eval} \mathfrak{g}^{nil}_{0,\infty}$ along the natural projection $\cQuad\rightarrow\mathfrak{g}_{0,\infty}^{nil}$, 
we get a $\gm$-equivariant isomorphism
\[
\mathcal{T}rip\cong W\times\cQuad.
\]
The statement about the relative dimension follows from the fact that $\Psi\in H^{0}(\mathbb{P}^{1},\text{ad}(\cE_\mu))$ is nilpotent if and only if the $\mathfrak{g}^{0}$-component \eqref{decomposition of g} of $\Psi$ is nilpotent, \eqref{global sections}, $eval(\Psi)=(\Psi,\Psi)$ for $\Psi\in\mathfrak{g}^{0}$ and by the fact that the evaluation map $H^{0}(\mathbb{P}^{1},\mathcal{O}(m))\rightarrow\mathbb{A}^{1}\times\mathbb{A}^{1}$, $\phi\mapsto(\phi_{0},\phi_{\infty})$ has nullity $m-1$ for $m>0$.
This finishes the proof of Lemma \ref{quad}.
\end{proof}
Thus we have reduced the problem of finding a stratification of $\mathcal{T}rip$ to finding a stratification of $\cQuad$.
\begin{remark}\label{twopointsonto}
\begin{enumerate}[(i)]
    \item The argument used above for reducing the problem from $\mathcal{T}rip$ to $\cQuad$ is a direct generalization of the argument given in \cite[Section 5.4]{Mel}.
    \item Let $S$ be a set of $k$-rational marked points on $\mathbb{P}^1$ such that $|S|>2$. Since the evaluation map $H^{0}(\mathbb{P}^{1},\mathcal{O}(m))\rightarrow(\mathbb{A}^{1})^{|S|}$, $\phi\mapsto(\phi_{x})_{x\in S}$ is not necessarily surjective for $m>0$, the proof of Lemma \ref{quad} fails for more than two marked points. 
    %\item From the description of $H^{0}(\mathbb{P}^{1},\text{ad}(\cE_\mu))^{nil}$ and $H^{0}(\mathbb{P}^{1},\text{ad}(\cE_\mu))$ in the proof of Lemma \ref{quad} and using Remark \ref{equivalent definition of Quad}, we see that $\frac{H^{0}(\mathbb{P}^{1},\text{ad}(\cE_\mu))^{nil}}{H^{0}(\mathbb{P}^{1},\text{ad}(\cE_\mu))}=\frac{1}{q^{\rk(G)}}$ and the same equality holds when instead of considering arbitrary global sections of $\text{ad}(\cE_\mu)$ on the left hand side, we consider global sections of $\text{ad}(\cE_\mu)$ compatible with parabolic structures at $0$ and $\infty$.
\end{enumerate}
\end{remark}
\subsection{Stratification of $\mathcal{B}_{J}^{\fin}$} \label{stratification of bjfin}
The following example of the Bialynicki--Birula decomposition will be important to us. 

Let $\mathbb{G}_{m}$ act on $G/P_{J}$ via $\mu$. 
We have an explicit description of the connected components of the fixed point locus given by the following result:

\begin{fact}\label{b2}
Recall from Section \ref{main theorem} that $L_{\mu}$ is the identity component of the centralizer of $\mu(\mathbb{G}_{m})$ in $G$.
Then
\[
(G/P_{J})^{\mathbb{G}_{m}}=\displaystyle\bigsqcup_{w\in W_{\Pi_{\mu}}\backslash W/W_{J}}\quad Z_{w}
\]
with $Z_{w}$ the orbit of $Ad_{w}(\mathfrak{p}_{J})$ under $L_{\mu}$. In particular, the connected components $Z_{w}$ of the fixed point locus $(G/P_{J})^{\mathbb{G}_{m}}$ appearing in the Bialynicki--Birula decomposition of $G/P_{J}$ (Fact \ref{bb1}(ii)) are in one to one correspondence with the elements of $W_{\Pi_{\mu}}\backslash W/W_{J}$.
\end{fact}
\par
The above statement follows by reducing to $\overline{k}$ and by noting that the proof of \cite[Lemma 1]{Fre} works for any algebraically closed field.

Note that $Z_{w}\cong L_{\mu}/(L_{\mu}\cap w\cdot P_{J})$ (see \cite[Proposition 7.12]{Mil1}), which is a partial flag variety of the Levi subgroup $L_{\mu}$ of $G$ defined over $k$. From Fact \ref{b2} we get:
\[
(G/P_{J})^{\mathbb{G}_{m}}\cong\displaystyle\bigsqcup_{w\in W_{\Pi_{\mu}}\backslash W/W_{J}}\quad L_{\mu}/(L_{\mu}\cap w\cdot P_{J}).
\]
Let $\pi: \mathcal{B}_{J}\rightarrow G/P_{J}$ be the projection.
Note that $\pi$ is $\gm$-equivariant where $\gm$ acts on $\mathcal{B}_{J}$ as in \eqref{action on B_J}. Thus $\mathcal{B}_{J}$ is an equivariant vector bundle over the smooth projective scheme $G/P_{J}$. By 
Proposition \ref{stratification of bundle}, we have a stratification of $\mathcal{B}_{J}^{\fin}$ by locally closed subsets as:
\begin{equation}\label{prpn2}
\mathcal{B}_{J}^{\fin}=\displaystyle\bigsqcup_{w\in W_{\Pi_{\mu}}\backslash W/W_{J}}\mathcal{B}_{J,w}^{+}
\end{equation}
and a decomposition of $\mathcal{B}_{J}^{\gm}$ as
\begin{equation}\label{relative dimension}
\mathcal{B}_{J}^{\gm}=\bigsqcup_{w\in W_{\Pi_{\mu}}\backslash W/W_{J}}V_{w,0},
\end{equation}
where $V_{w,0}$ are the connected components of $\mathcal{B}_{J}^{\gm}$, $w\in W_{\Pi_{\mu}}\backslash W/W_{J}$. 
Moreover, there are affine fibrations $lim_{w}:\mathcal{B}_{J,w}^{+}\rightarrow V_{w,0}$ given by the limit map as $t\to 0$.
\begin{remark}
We can describe $V_{w,0}$ more explicitly, it is isomorphic to $\mathcal{B}_{\Pi_{\mu}\cap w\cdot J}$, where the underlying group is $L_{\mu}$.
Indeed, identify $L_\mu/(L_{\mu}\cap w\cdot P_{J})$ with the scheme of parabolic subgroups of $L_{\mu}$ that are conjugate to $L_{\mu}\cap w\cdot P_{J}$. By Fact \ref{b2}, we obtain
\begin{equation}\label{pre-steinberg}
V_{w,0}\cong\{(\mathfrak{p}',v):\mathfrak{p}'\in L_\mu/(L_{\mu}\cap w\cdot P_{J}),v\in\mathfrak{p}'\},
\end{equation}
where the above isomorphism is given by $(\mathfrak{p},v)\mapsto(\mathfrak{p}\cap \lie(L_{\mu}),v)$. Note that if, for some $v'\in\mathfrak{g}$ we have $\text{Ad}_{\mu(t)}\cdot v'=v'$ for all $t\in\gm$, then $v'\in\lie(L_\mu)$. Therefore, $v\in\mathfrak{p}\cap\lie(L_{\mu})=\lie(P\cap L_{\mu})$. Thus $(\mathfrak{p},v)\mapsto(\mathfrak{p}\cap \lie(L_{\mu}),v)$ is a well-defined morphism.
\end{remark}
The next proposition gives the relative dimension of $lim_{w}$.
\begin{proposition}\label{relative dimension of lim}
The relative dimension of the affine fibration $lim_{w}:\mathcal{B}_{J,w}^{+}\rightarrow V_{w,0}$ is $(\dim G-\dim L_{\mu})/2$.
\end{proposition}
\begin{proof}
To calculate the relative dimension of $\clz:\mathcal{B}_{J,w}^{+}\rightarrow V_{w,0}$, we will use Fact \ref{bb1}(iii) on $\overline{\mathcal{B}}_{J}$ (this gives us the desired relative dimension because 
%by Proposition \ref{prpn1},
$\clz$ is obtained by base change of the affine fibration that we get by applying the Bialynicki--Birula decomposition on $\overline{\mathcal{B}}_{J}$). 

Let $\underline{a}=(Ad_{w}(\mathfrak{p}_{J}),0)\in V_{w,0}(k)$.
Since $\underline{a}$ is a $\mathbb{G}_{m}$-fixed point (see Fact \ref{b2}), we get an action %(see \eqref{quadgm}),
\[
\mathbb{G}_{m}\curvearrowright T_{\ca}(\overline{\mathcal{B}}_{J})=T_{\ca}(\mathcal{B}_{J}),
\]
where $T_{\ca}(\overline{\mathcal{B}}_{J})=T_{\ca}(\mathcal{B}_{J})$ because $\mathcal{B}_{J}$ is an open subscheme of $\overline{\mathcal{B}}_{J}$. 
%This gives a decomposition of $T_{\ca}(\mathcal{B}_{J})$ into positive, negative and fixed eigenspaces. 
Let $T_{\ca}^{+}(\mathcal{B}_{J})$ (resp. $T_{\ca}^{-}(\mathcal{B}_{J})$) denote the positive (resp. negative) eigenspace of the $\gm$-action on the tangent space of $\mathcal{B}_{J}$ at $\ca$ and let $T_{\ca}^{0}(\mathcal{B}_{J})$ denote the fixed eigenspace of the $\gm$-action on the tangent space of $\mathcal{B}_{J}$ at $\ca$.
Since $\underline{a}\in V_{w,0}(k)$, the relative dimension of the affine fibration $lim_{w}:\mathcal{B}_{J,w}^{+}\rightarrow V_{w,0}$ is equal to $\dim T_{\ca}^{+}(\mathcal{B}_{J})$ by Fact \ref{bb1}(iii), so it suffices to
calculate $\dim T_{\ca}^{+}(\mathcal{B}_{J})$.
Note that $T_{\ca}(\mathcal{B}_{J})$ is $\gm$-equivariantly isomorphic to 
$\mathfrak{g}/Ad_{w}(\mathfrak{p}_{J})\oplus Ad_{w}(\mathfrak{p}_{J})$. Since $L_{\mu}$ is in the centralizer of $\mu(\gm)$, we see that $\text{Ad}_{\mu(t)}$ acts on $\text{Ad}_{w}(\mathfrak{g}_{\alpha})$ via multiplication by $t^{\langle w\cdot\alpha,\mu\rangle}$,
$t\in\gm$, $\alpha\in\Phi$
and 
acts trivially on $\text{Ad}_{w}(\mathfrak{h})=\mathfrak{h}$. Thus, $T_{\ca}(\mathcal{B}_{J})$ is $\gm$-equivariantly isomorphic to $\mathfrak{g}$, which gives
\[
\dim T_{\ca}^{+}(\mathcal{B}_{J})
=(\dim G-\dim L_{\mu})/2.
\]
\end{proof}
\subsection{Stratification of $\cQuad$}\label{Stratification of quad}
We will now work towards obtaining a 
stratification of $\cQuad$ by using the stratification \eqref{prpn2} of $\mathcal{B}_{J}^{\fin}$.
Once we have such a stratification, Theorem \ref{trip} will be an easy consequence as explained in Section \ref{strategy of the proof}.
%First, let us show that the
%$\cQuad$ 
%is contained in 
%$\mathcal{B}_{J_{0}}^{\fin}\times\mathcal{B}_{J_{\inft%y}}^{\fin}$.
\begin{lemma}\label{lemma}
%Keep notations as above. 
We have $\cQuad\subset\cB^{\fin}\times\cBi^{\fin}$.
\end{lemma}
\begin{proof}
It is enough to show that $\cQuad$ is contained in the constructible subset $\cB^{\fin}\times\cBi^{\fin}$ at the level of closed points. Let $K$ be a finite extension of $k$. Let $(\mathfrak{p}_{0},v_0,\mathfrak{p}_{\infty},v_\infty)\in \cQuad(K
)$, then $(\mathfrak{p}_{0},v_{0})\in\cB(K
)$ and $(\mathfrak{p}_{\infty},v_{\infty})\in\cBi(K)$.
The lemma will follow if we show $\lim_{t\rightarrow 0}t\cdot(\mathfrak{p}_{x},v_{x})$ exists in $\mathcal{B}_{J_{x}}$, $x=0,\infty$.

Since $G/P_{J_{x}}$ is a projective scheme, we get that $\lim_{t\to 0}t\cdot \mathfrak{p}_{x}$ exists, $x=0,\infty$. By definition of $\cQuad$, $\mathfrak{g}^{-}$-component \eqref{decomposition of g} of $v_{x}$ is $0$ and therefore $\lim_{t\to 0}t\cdot v_{x}$
exists and is equal to the $\mathfrak{g}^{0}$-component \eqref{decomposition of g} of $v_{x}$, $x=0,\infty$. 
Thus, $\lim_{t\to 0}t\cdot(\mathfrak{p}_{0},v_0)$ exists in $\mathcal{B}_{J_{0}}$ and $\lim_{t\to 0}t\cdot(\mathfrak{p}_{\infty},v_\infty)$ exists in $\mathcal{B}_{J_{\infty}}$.  
\end{proof}
Recall $W_{\Pi_{\mu}},W_{J_{0}},W_{J_{\infty}},L_{\mu}$ from Section \ref{main theorem}. 
For $w\in W_{\Pi_{\mu}}\backslash W/W_{J_{0}}$, $w'\in W_{\Pi_{\mu}}\backslash W/W_{J_{\infty}}$, recall $V_{w,0}$, $V_{w',0}$, $lim_{w}$, $lim_{w'}$ from Section \ref{stratification of bjfin} and put
\[
\cQuad_{w,w'}^{\gm}:=(V_{w,0}\times V_{w',0})\cap\cQuad.
\] 
Let $\cQuad_{w,w'}^{+}$ be the pullback of
$\clz\times\cli:\mathcal{B}^{+}_{J_{0},w}\times\mathcal{B}^{+}_{J_{\infty},w'}\longrightarrow V_{w,0}\times V_{w',0}$ along $\cQuad_{w,w'}^{\gm}\rightarrow V_{w,0}\times V_{w',0}$, that is, we have the following cartesian square:
\[
\begin{tikzcd}
& \cQuad_{w,w'}^{+} \arrow{r}{} \arrow{d}[swap]{}
& \mathcal{B}_{J_{0},w}^{+}\times\mathcal{B}_{J_{\infty},w'}^{+} \arrow{d}{\clz\times\cli} \\
& \cQuad_{w,w'}^{\gm} \arrow[swap]{r}{}
& V_{w,0}\times V_{w',0}
\end{tikzcd}
\]
Let us denote the left vertical arrow in the above diagram again by $\clz\times\cli$.
Next, we would like to show that the schemes $\cQuad_{w,w'}^{+}$ give a stratification of $\cQuad$, which is the content of the next lemma.
\begin{lemma}\label{quadquad}
%Keep notations as above. We have  We 
The scheme $\cQuad_{w,w'}^{+}\subset\mathcal{B}^{+}_{J_{0},w}\times\mathcal{B}^{+}_{J_{\infty},w'}$ is contained in $\cQuad$.
\end{lemma}
\begin{proof}
It is enough to show that $\cQuad^{+}_{w,w'}$ is contained in $\cQuad$ at the level of closed points. Let $K$ be a finite extension of $k$. Let $(\mathfrak{p}_{0},v_0,\mathfrak{p}_{\infty},v_\infty)\in \cQuad_{w,w'}^{+}(K)$, then we have $\lim_{t\rightarrow 0}t\cdot(\mathfrak{p}_{0},v_{0},\mathfrak{p}_{\infty},v_{\infty})\in\cQuad_{w,w'}^{\gm}(K)$. In particular, $\lim_{t\to 0}t\cdot v_{x}$ exists in $\mathfrak{g}$ and is equal to the $\mathfrak{g}^{0}$-component \eqref{decomposition of g} of $v_{x}$, $x=0,\infty$.
Since for any $v\in\mathfrak{g}$, $\lim_{t\to 0}t\cdot v$ exists in $\mathfrak{g}$ if and only if the $\mathfrak{g}^{-}$-component \eqref{decomposition of g} of $v$ is zero, the $\mathfrak{g}^{-}$-components of $v_{0}$ and $v_{\infty}$ are zero.
Moreover, since $\cQuad_{w,w'}^{\gm}$ is the closed subscheme of $\cQuad$ consisting of quadruples $(\mathfrak{p}_{0},n,\mathfrak{p}_{\infty},n)$ such that
$\mathfrak{p}_{0}\in Z_{w}$ (resp. $\mathfrak{p}_{\infty}\in Z_{w'}$), $n$ is a nilpotent element of $\mathfrak{g}$ such that $n\in\mathfrak{p}_{0}$ and $n\in\mathfrak{p}_{\infty}$, we get that the $\mathfrak{g}^{0}$-components \eqref{decomposition of g} of $v_{0}$ and $v_{\infty}$ are equal and nilpotent.
The lemma now follows from Remark \ref{equivalent definition of Quad}. 
\end{proof}
Recall the generalized Steinberg varieties from Section \ref{coproduct}.
%The following lemma
%identifies the schemes $\cQuad_{w,w'}^{\gm}$
%with the generalized Steinberg varieties.
\begin{lemma}\label{identifying quad with steinberg}
%Keep notations as above. 
The schemes $\cQuad_{w,w'}^{\gm}$ are isomorphic to the generalized Steinberg varieties $St_{L_{\mu}}(\Pi_{\mu}\cap w\cdot J_{0},\Pi_{\mu}\cap w'\cdot J_{\infty})$, $w\in D^{G}_{\Pi_{\mu},J_{0}}, w'\in D^{G}_{\Pi_{\mu},J_{\infty}}$.
\end{lemma}
\begin{proof}
Notice that $\cQuad_{w,w'}^{\gm}$ is the closed subscheme of $\cQuad$ consisting of quadruples $(\mathfrak{p}_{0},n,\mathfrak{p}_{\infty},n)$ such that
$\mathfrak{p}_{0}\in Z_{w}$ (resp. $\mathfrak{p}_{\infty}\in Z_{w'}$),
$n$ is a nilpotent element of $\mathfrak{g}$ such that $n\in\mathfrak{p}_{0}$ and $n\in\mathfrak{p}_{\infty}$
(note that the $\mathfrak{g}^{+}$ and $\mathfrak{g}^{-}$-components of $n$ are $0$ since $\gm$ acts trivially on $\cQuad_{w,w'}^{\gm}$). Thus we have
\begin{equation}\label{isomorphism to steinberg variety}
\cQuad_{w,w'}^{\gm}\cong St_{L_{\mu}}(\Pi_{\mu}\cap w\cdot J_{0},\Pi_{\mu}\cap w'\cdot J_{\infty}),
\end{equation}
where the above isomorphism is given by $(\mathfrak{p}_{0},n,\mathfrak{p}_{\infty},n)\mapsto(n,\mathfrak{p}_{0}\cap \lie(L_{\mu}),\mathfrak{p}_{\infty}\cap \lie(L_{\mu}))$.
\end{proof}
Next, we show that the generalized Steinberg varieties are connected.
\begin{lemma}\label{connectedness of steinberg varieties}
Recall $\Pi_{H}, J_{1}$ from $J_{2}$
%and $St_{H}(J_{1},J_{2})$
from Section \ref{coproduct}. Then $St_{H}(J_{1},J_{2})$ is connected.
\end{lemma}
\begin{proof}
We show that $St_{H}(J_{1},J_{2})$ is geometrically connected, that is, $St_{H}(J_{1},J_{2})_{K}$ is connected where $K$ is the algebraic closure of $k$. Note that since the natural projection $St_{H}(J_{1},J_{2})_{K}\rightarrow St_{H}(J_{1},J_{2})$ is surjective,
%as surjective morphisms are preserved under base change \cite[Lemma 29.9.4]{Sta}. Thus
we will have that $St_{H}(J_{1},J_{2})$ is connected.

Since $H/P_{J_{1}}\times H/P_{J_{2}}$ is geometrically connected (see \cite[Proposition 5.2.4]{Poo} and use the fact that quotient commutes with field extensions) and $St_{H}(J_{1},J_{2})_{K}$ is conical over
$(H/P_{J_{1}}\times H/P_{J_{2}})_{K}$, it follows that $St_{H}(J_{1},J_{2})_{K}$ is connected.

\end{proof}
Thus by Lemma \ref{lemma} and Lemma \ref{quadquad} we get a stratification of $\cQuad$ by locally closed subsets as:
\begin{equation}\label{strata of quad}
\cQuad \; =\displaystyle\bigsqcup_{\substack{w\in W_{\Pi_{\mu}}\backslash W/W_{J_{0}}
\\
w'\in W_{\Pi_{\mu}}\backslash W/W_{J_{\infty}}}}\cQuad_{w,w'}^{+}
\end{equation}
and a decomposition of the fixed point locus $\cQuad^{\mathbb{G}_{m}}$ as:
\[
\cQuad^{\gm}=\displaystyle\bigsqcup_{\substack{w\in W_{\Pi_{\mu}}\backslash W/W_{J_{0}}
\\
w'\in W_{\Pi_{\mu}}\backslash W/W_{J_{\infty}}}} \cQuad_{w,w'}^{\gm},
\]     
where $\cQuad_{w,w'}^{\gm}$ are the connected components (see Lemma \ref{identifying quad with steinberg} and Lemma \ref{connectedness of steinberg varieties}) of $\cQuad^{\gm}$. Moreover, we have affine fibrations
\[
\clz\times\cli:\cQuad_{w,w'}^+\rightarrow \cQuad_{w,w'}^{\gm}.
\]
%Finally we calculate the relative dimension of the affine fibration $\clz\times\cli:\cQuad_{w,w'}^+\rightarrow \cQuad_{w,w'}^{\gm}$.
\begin{corollary}\label{bb2}
The relative dimension of the above affine fibration
%\[
%\clz\times\cli:\cQuad_{w,w'}^{+}\rightarrow \cQuad_{w,w'}^{\gm}
%\]
is equal to $\dim G-\dim L_\mu$.
\end{corollary}
\begin{proof}
Since the above affine fibration %$\clz\times\cli:\cQuad_{w,w'}^{+}\rightarrow \cQuad_{w,w'}^{\gm}$ 
is obtained by base change of the affine fibration $\clz\times\cli:\mathcal{B}_{J_{0},w}^{+}\times\mathcal{B}_{J_{\infty},w'}^{+}\rightarrow V_{w,0}\times V_{w',0}$, the corollary follows from Proposition \ref{relative dimension of lim}.
\end{proof}
\subsection{Completing the proof of Theorem \ref{trip}}
Recall $ev^{0,\infty}$ defined in Lemma \ref{quad}. For each $w\in W_{\Pi_{\mu}}\backslash W/W_{J_{0}}$, $w'\in W_{\Pi_{\mu}}\backslash W/W_{J_{\infty}}$, put
\[
\mathcal{T}rip_{w,w'}^{+}:=\big(ev^{0,\infty}\big)^{-1}\big(\cQuad^{+}_{w,w'}\big).
\]
Since $\cQuad^{+}_{w,w'}$, $w\in W_{\Pi_{\mu}}\backslash W/W_{J_{0}}$, $w'\in W_{\Pi_{\mu}}\backslash W/W_{J_{\infty}}$ form a stratification of $\cQuad$, we get a stratification of $\mathcal{T}rip$ by locally closed subsets as:
\[
\mathcal{T}rip=\displaystyle\bigsqcup_{\substack{w\in W_{\Pi_{\mu}}\backslash W/W_{J_{0}}
\\
w'\in W_{\Pi_{\mu}}\backslash W/W_{J_{\infty}}}}\mathcal{T}rip_{w,w'}^{+}.
\]
Now let $\big(ev^{0,\infty}\big)_{w,w'}:=ev^{0,\infty}|_{\mathcal{T}rip_{\mu}(J_{0},J_{\infty})_{w,w'}^{+}}$. Consider the morphism
\[
\big(\clz\times\cli\big)\circ\big(ev^{0,\infty}\big)_{w,w'}:\mathcal{T}rip_{w,w'}^{+}\rightarrow
\mathcal{T}rip_{w,w'}^{\gm},
\]
where $\mathcal{T}rip_{w,w'}^{\gm}$ is the subscheme of $\mathcal{T}rip^{\gm}$ that is identified with $\cQuad_{w,w'}^{\gm}$ via \eqref{isomorphism}. The last part of the proof of Theorem \ref{trip} follows from the next lemma.
%Now we will need the following result (see \cite[Proposition 5.2]{Ra}), which describes $\aut(\cE_{\mu})$ as a scheme:
%\begin{fact}\label{aut of emu}
%Let $\cE_{\mu}$ be as above. Then $\aut(\cE_{\mu})$ is isomorphic as a scheme to
%\[
%L_{\mu}\times \prod_{\alpha\in\Phi:\langle\alpha,\mu\rangle<0}H^{0}(\mathbb{P}^{1},\mathcal{O}(-\langle\alpha,\mu\rangle)).
%\]
%\end{fact}
\begin{lemma}\label{fiber}
The morphism $\big(\clz\times\cli\big)\circ\big(ev^{0,\infty}\big)_{w,w'}$ is an affine fibration of relative dimension $\dim(\aut(\cE_{\mu}))-\dim(L_{\mu})$.
\end{lemma}
\begin{proof}
Since $\clz\times\cli$ is a trivial affine fibration (see Lemma \ref{quad}) and $\big(ev^{0,\infty}\big)_{w,w'}$ is an affine fibration (see Corollary \ref{bb2}), their composition $\big(\clz\times\cli\big)\circ\big(ev^{0,\infty}\big)_{w,w'}$ is an affine fibration.

Now let us calculate the required relative dimension. By Fact \ref{aut of emu}, we have
\[
\dim(\aut(\cE_{\mu}))-\dim(L_{\mu})
=
\dim\Big(\prod_{\alpha\in\Phi:\langle\alpha,\mu\rangle>0}H^{0}(\mathbb{P}^{1},\mathcal{O}(\langle\alpha,\mu\rangle))\Big)
=
\sum_{\langle\alpha,\mu\rangle>0}\Big(\langle\alpha,\mu\rangle+1\Big).
\]
Combining Lemma \ref{quad} with Corollary \ref{bb2}, we see that $\big(\clz\times\cli\big)\circ\big(ev^{0,\infty}\big)_{w,w'}$ is of relative dimension 
\[
    \dim G-\dim L_\mu+\sum_{\langle\alpha,\mu\rangle>0}\Big(\langle\alpha,\mu\rangle-1\Big).
\]
Now the lemma follows by noting that $\dim G-\dim L_\mu=2|\{\alpha\in\Phi:\langle\alpha,\mu\rangle>0\}|$.
\end{proof}
\begin{remark}
Consider the case when $k=\mathbb{F}_{q}$. Define a nilpotent parabolic pair of type $(G,\mathbb{P}^{1},\{0,\infty\})$ to be a collection $(\cE,\mathfrak{p}_{0},\mathfrak{p}_{\infty},\Psi),$ where $\cE$ is a principal $G$-bundle over $\mathbb{P}^{1}$, $\mathfrak{p}_{x}$ is a parabolic structure on $\cE$ at $x$, $\Psi$ is a nilpotent section of $\text{ad}(\cE)$ such that $\Psi_{x}\in\mathfrak{p}_{x}$, $x=0,\infty$. 
We will denote the groupoid of nilpotent parabolic pairs by $\mathcal{P}air^{nilp}(G,\mathbb{P}^{1},\{0,\infty\})$. Then $\mathcal{P}air^{nilp}(G,\mathbb{P}^{1},\{0,\infty\})$ decomposes into subgroupoids according to the type of parabolic structures at $0$ and $\infty$. We denote these subgroupoids by $\mathcal{P}air^{nilp}_{J_{0},J_{\infty}}(G,\mathbb{P}^{1},\{0,\infty\})$, $J_{0},J_{\infty}\subset\Pi$. For $\mu\in X_{+}(T)$, let $\mathcal{P}air^{nilp,\mu}_{J_{0},J_{\infty}}(G,\mathbb{P}^{1},\{0,\infty\})$ denote the subgroupoid of $\mathcal{P}air^{nilp}_{J_{0},J_{\infty}}(G,\mathbb{P}^{1},\{0,\infty\})$ such that the underlying principal $G$-bundle over $\mathbb{P}^1$ is isomorphic to $\cE_\mu$. Explicitly knowing $|\mathcal{T}rip_{\mu}(J_{0},J_{\infty})|$ allows us to calculate the volume of the groupoid $\mathcal{P}air^{nilp,\mu}_{J_{0},J_{\infty}}(G,\mathbb{P}^{1},\{0,\infty\})$. More concretely, we have 
\[
[\mathcal{P}air^{nilp,\mu}_{J_{0},J_{\infty}}(G,\mathbb{P}^{1},\{0,\infty\})]=\frac{|\mathcal{T}rip_{\mu}(J_{0},J_{\infty})|}{|\aut(\cE_{\mu})|},
\]
where $[\mathfrak{X}]$ denotes the volume of a groupoid $\mathfrak{X}$.
\end{remark}
\subsection{Comparison between different groups}\label{comparison}
In this section, we let $k=\mathbb{F}_{q}$. We will compare $|\mathcal{T}rip_{\mu}(J_{0},J_{\infty})|$ for different groups below. For this, we introduce the following notation.
\begin{notation}\label{comparison2} Recall $H,T_{H},B_{H},\Pi_{H}$ from Section \ref{coproduct}. Let $\nu\in X_{+}(T_{H})$ and let $\cE_{\nu}$ denote the principal $H$-bundle over $\mathbb{P}^1$ induced by $\nu$. For $J_{0},J_{\infty}\subset\Pi_{H}$, as before we let $\mathcal{T}rip_{\mu,H}(J_{0},J_{\infty})$ denote the scheme parameterizing triples $(\mathfrak{p}_{0},\mathfrak{p}_{\infty},\Psi)$ such that $\Psi$ is a nilpotent section of ad$(\cE_{\nu})$, $\mathfrak{p}_{0}$ (resp. $\mathfrak{p}_{\infty}$) is a parabolic structure at $0$ (resp. $\infty$) of type $J_{0}$ (resp. $J_{\infty}$) and $\Psi_{0}\in \mathfrak{p}_{0}$,  $\Psi_{\infty}\in \mathfrak{p}_{\infty}$. Again as before, define $[\mathcal{T}rip_{\nu,H}]:\mathcal{P}(\Pi_{H})\times\mathcal{P}(\Pi_{H})\rightarrow\mathbb{Z}$ by
\[
[\mathcal{T}rip_{\nu,H}](J_{0},J_{\infty}):=|\mathcal{T}rip_{\nu,H}(J_{0},J_{\infty})|
\]
\end{notation}
Consider the following two situations:
\begin{enumerate}[(i)]
\item Recall $G,T,B,\Pi$ from Section \ref{parabolic-levi}. Let $G':=[G,G]$ be the derived group of $G$. Let $j:G'\rightarrow G$ be the natural inclusion. Denote the split maximal torus $T\cap G'$ of $G'$ by $T'$ and the Borel $\mathbb{F}_{q}$-subgroup $B\cap G'$ of $G'$ by $B'$. Let $\mu'\in X_{+}(T')$, we have $\mu:=j\circ\mu'\in X_{+}(T)$. Since the root systems of $G$ and $G'$ are isomorphic, we will consider $[\mathcal{T}rip_{\mu',G'}]$ and $[\mathcal{T}rip_{\mu,G}]$ as functions with domain $\Pi\times\Pi$.
\item Recall that a morphism $u:G_{1}\rightarrow G_{2}$ of connected affine algebraic groups over $\mathbb{F}_{q}$ is called a central isogeny if it is a flat surjection such that $\ker(u)$ is a finite $\mathbb{F}_{q}$-group and is central in $G_{1}$ (see \cite[Definition 18.1]{Mil1}). Now let $u:G_{1}\rightarrow G_{2}$ be a central isogeny of split connected reductive groups over $\mathbb{F}_{q}$. Let $T_{1}$ be a split maximal torus of $G_{1}$ and let $B_{1}$ be a Borel $\mathbb{F}_{q}$-subgroup of $G_{1}$ containing $T_{1}$. Then $T_{2}:=u(T_{1})$ is a split maximal torus of $G_{2}$ and $B_{2}:=u(B_{1})$ is a Borel $\mathbb{F}_{q}$-subgroup of $G_{2}$ containing $T_{2}$ (see \cite[Proposition 17.20]{Mil1}).
Let $\mu_{1}\in X_{+}(T_1)$, we have $\mu_{2}:=u\circ\mu_{1}\in X_{+}(T_{2})$. Since the root systems of $G_1$ and $G_2$ are isomorphic, we will consider $[\mathcal{T}rip_{\mu_{1},G_{1}}]$ and $[\mathcal{T}rip_{\mu_{2},G_{2}}]$ as functions with domain $\Pi_{1}\times\Pi_{1}$, where $\Pi_{1}$ is the set of simple roots of $G_1$ with respect to $(B_{1},T_{1})$. 
\end{enumerate}
We have the following:
\begin{corollary}\label{slngln}
\begin{enumerate}[(a)]
    \item With notations as in $(i)$ above, we have
\[
[\mathcal{T}rip_{\mu',G'}]= [\mathcal{T}rip_{\mu,G}].
\]
\item With notations as in $(ii)$ above, we have
\[
[\mathcal{T}rip_{\mu_{1},G_{1}}]= [\mathcal{T}rip_{\mu_{2},G_{2}}].
\]
\end{enumerate}
\end{corollary}
%We give the proof of Corollary \ref{slngln} in Section \ref{formula}.
As special cases, we may take $G=GL_n$ and $G'=SL_n$ in Corollary \ref{slngln} $(a)$ and $SL_n\rightarrow PGL_n\cong SL_{n}/\mu_{n}$ or $Spin_{n}\rightarrow SO_{n}$ in Corollary \ref{slngln} $(b)$.
%\subsection{Proof of Corollary \ref{slngln}} In this section, we will give the proof of Corollary \ref{slngln}. 

Now we give a proof of Corollary \ref{slngln} $(a)$.
First we need the following notation:
\begin{notation}
 For any affine algebraic group $H$ over $\mathbb{F}_{q}$ and a cocharacter $\mu$ of a maximal torus, we will denote the centralizer of $\mu(\gm)$ in $H$ by $Z_{H}(\mu)$.
\end{notation}
Let $\mu$, $\mu'$ be as in Notation \ref{comparison2} $(i)$ above. 
%We have the following lemma.
%We need a lemma:
\begin{lemma}\label{derived group of levi}
%Keep notations as in Section \ref{comparison}. Then 
We have the following equality
\[
[L_{\mu},L_{\mu}]=[L_{\mu'},L_{\mu'}].
\]
In particular, the root systems of $L_{\mu}$ and $L_{\mu'}$ are isomorphic.
\end{lemma}
\begin{proof}
We have $L_{\mu}=Z_{G}(\mu)^{\circ}$ and
\begin{equation}\label{derivedlevi-centralizer}
    L_{\mu'}=Z_{G'}(\mu')^{\circ}=(Z_{G}(\mu)\cap G')^{\circ}=(L_{\mu}\cap G')^{\circ}
\end{equation}
Clearly by $\eqref{derivedlevi-centralizer}$, we have $[L_{\mu'},L_{\mu'}]\subset[L_{\mu},L_{\mu}]$. Now we show the other inclusion. Since $G'=[G,G]$, 
we have 
$[L_{\mu},L_{\mu}]\subset G'$. Thus we have
$
[[L_{\mu},L_{\mu}],[L_{\mu},L_{\mu}]]\subset[G',G']$.
Since the derived group of any connected reductive group over $\mathbb{F}_{q}$ is perfect (see \cite[Proposition 1.2.6]{CGP}), $[[L_{\mu},L_{\mu}],[L_{\mu},L_{\mu}]]=[L_{\mu},L_{\mu}]$ and hence
$
[L_{\mu},L_{\mu}]\subset[G',G'].
$
Combining it with the fact that $[L_\mu,L_\mu]$ is connected, we get that
\[
[L_\mu,L_\mu]\subset(L_\mu\cap G')^{\circ}
\]
Now $\eqref{derivedlevi-centralizer}$ gives us that $[L_\mu,L_\mu]\subset L_{\mu'}$ and hence we have the other inclusion $[L_\mu,L_\mu]\subset [L_{\mu'},L_{\mu'}]$. This finishes the proof of Lemma~\ref{derived group of levi}.
\end{proof}
We return to the proof of Corollary \ref{slngln} $(a)$. By Lemma~\ref{derived group of levi}, root systems of $L_\mu$ and $L_{\mu'}$ are isomorphic, which gives us that $[Sp_{L_{\mu}}]=[Sp_{L_{\mu'}}]$ as the number of points of the generalized Springer varieties depend only on the root system of the underlying affine algebraic group (see Theorem \ref{steinberg}$(i)$), hence $\Delta_{L_{\mu}}([Sp_{L_{\mu}}])=\Delta_{L_{\mu'}}([Sp_{L_{\mu'}}])$ (see the definition of coproduct in \ref{coproduct}). Now Corollary \ref{slngln} $(a)$ follows from the equality $\dim(\aut(\cE_{\mu}))-\dim(L_{\mu})=\dim(\aut(\cE_{\mu'}))-\dim(L_{\mu'})$ (see Theorem \ref{trip} and Fact \ref{aut of emu}).

Now we give a proof of Corollary \ref{slngln} $(b)$. We have $u_{|_{L_{\mu_{1}}}}:L_{\mu_{1}}\rightarrow L_{\mu_{2}}$ is a flat surjective morphism (see \cite[Corollary 2.1.9]{CGP}). Moreover, $\ker(u_{|_{L_{\mu_{1}}}}:L_{\mu_{1}}\rightarrow L_{\mu_{2}})$ is finite.
%as the restriction of a finite morphism to closed subschemes is again a finite morphism. 
Clearly, $\ker(u_{|_{L_{\mu_{1}}}})$ is central in $L_{\mu_{1}}$ as $\ker(u)$ is central in $G_{1}$. Hence, $u_{|_{L_{\mu_{1}}}}:L_{\mu_{1}}\rightarrow L_{\mu_{2}}$ is a central isogeny. So we get that the root systems of $L_{\mu_{1}}$ and $L_{\mu_{2}}$ are isomorphic (see \cite[Proposition 23.5]{Mil1}), which gives us that $[Sp_{L_{\mu_{1}}}]=[Sp_{L_{\mu_{2}}}]$ as the number of points of the generalized Springer varieties depend only on the root system of the underlying affine algebraic group (see Theorem \ref{steinberg}$(i)$), hence $\Delta_{L_{\mu_{1}}}([Sp_{L_{\mu}}])=\Delta_{L_{\mu_{2}}}([Sp_{L_{\mu'}}])$ (see the definition of coproduct in \ref{coproduct}). Now Corollary \ref{slngln} $(b)$ follows from the equality $\dim(\aut(\cE_{\mu_{1}}))-\dim(L_{\mu_{1}})=\dim(\aut(\cE_{\mu_{2}}))-\dim(L_{\mu_{2}})$ (see Theorem \ref{trip} and Fact \ref{aut of emu}).

%This finishes the proof of Corollary \ref{slngln} $(b)$.
\qed
\section{Special case of vector bundles over \texorpdfstring{$\mathbb{P}^1$}{P1}}\label{case of gln}
In this section, we work over $k=\mathbb{F}_{q}$ and derive the Mellit's result \cite[Section 5.4]{Mel} using our method. 
\subsection{Symmetric functions}Let us recall the notions of lambda rings, plethystic substitutions, and plethystic exponentials from \cite[Section 2.1, Section 2.2]{Mel}.

Fix a base ring $R$. We will denote the ring of symmetric functions in the sequence of variables $(x_{1},x_{2},\ldots)$ with coefficients in $R$ by $\text{Sym}_{R}[X]$, where $X=(x_{1},x_{2},\ldots)$.
By $\Sym_{R}[X,Y]$, we denote the ring of functions in the two sequences
of variables $X=(x_{1},x_{2},\ldots)$ and $Y=(y_{1},y_{2},\ldots)$, preserved under simultaneous permutation.
%We will denote the ring of symmetric functions with coefficients in $R$ that are symmetric in the two sequences of variables $X=(x_{1},x_{2},\ldots)$ and $Y=(y_{1},y_{2},\ldots)$ by $\text{Sym}_{R}[X,Y]$. 
All the variables in this paper have degree $1$. We will denote the degree $n$ component of $\Sym_{R}[X]$ by $\Sym_{R}^{n}[X]$ and we will denote the bidegree $(n,n)$ component of $\Sym_{R}[X,Y]$ by $\Sym_{R}^{n}[X,Y]$.
\begin{definition}
Let $\Lambda$ be a ring such that $\mathbb{Q}\subset \Lambda$. A lambda ring structure on $\Lambda$ is a collection of ring homomorphisms $p_{n}:\Lambda\rightarrow \Lambda$, $n\in\mathbb{Z}_{>0}$ satisfying:
\begin{enumerate}
    \item $p_{1}(x)=x$, $x\in \Lambda$ and
    \item $p_{m}(p_{n}(x))=p_{mn}(x)$, $m,n\in\mathbb{Z}_{>0}$, $x\in \Lambda$.
\end{enumerate}
By a lambda ring, we will mean a ring together with a lambda ring structure.

%In the above defintion, we require $\Lambda$ to contain $\mathbb{Q}$ because of the fact that $\{p_{\lambda}:\lambda\text{ is a partition}\}$ forms a basis of the ring of symmetric functions when the base ring contains $\mathbb{Q}$.
\end{definition}
When our base ring $R$ is itself a lambda ring, we define a lambda ring structure on $\text{Sym}_{R}[X]$ 
as follows: note that our ring is freely generated as an $R$-algebra by the Newton polynomials $p_m(X)$ since $R\supset\mathbb{Q}$. Thus for each $n\in\mathbb{Z}_{>0}$, there is a unique homomorphism $p_n: \Sym_R[X]\rightarrow \Sym_R[X]$ whose restriction to $R$ is given by the lambda ring structure on $R$ and $p_n(p_m(X))=p_{nm}(X)$ for all $m\in\mathbb{Z}_{>0}$. We define the lambda ring structure on $\Sym_R[X,Y]$ similarly.

The lambda ring structure that we consider on $\mathbb{Q}(q)$ is defined as:
\[
p_{n}:\mathbb{Q}(q)\rightarrow\mathbb{Q}(q), \quad n\in\mathbb{N}
\]
\[
r\mapsto r, \quad q\mapsto q^{n},\quad r\in\mathbb{Q}.
\]
The lambda ring structure that we consider on $\mathbb{Q}[[q^{-1}]]$ is defined as:
\[
p_{n}:\mathbb{Q}[[q^{-1}]]\rightarrow\mathbb{Q}[[q^{-1}]], \quad n\in\mathbb{N}
\]
\[
r\mapsto r, \quad q^{-1}\mapsto q^{-n},\quad r\in\mathbb{Q}.
\]
The lambda ring structure that we consider on $\mathbb{Q}[[q^{-1}]][[t]]$ is defined as:
\[
p_{n}:\mathbb{Q}[[q^{-1}]][[t]]\rightarrow\mathbb{Q}[[q^{-1}]][[t]] , \quad n\in\mathbb{N}
\]
\[
r\mapsto r, \quad q^{-1}\mapsto q^{-n},\quad t\mapsto t^n,\quad r\in\mathbb{Q}.
\]
%Note that $\mathbb{Q}(q)[[t]]$ is a sub lambda ring of $\mathbb{Q}[[q^{-1}]][[t]]$.
\begin{definition}
Let $\Lambda$ be a lambda ring.
Let $F\in\text{Sym}_{\mathbb{Q}}[X]$ and $x\in\Lambda$. We define the plethystic action of $F$ on $x$ as follows: write $F$ as a polynomial in power sum symmetric functions, say $F=f(p_{1},p_{2},\ldots)$ for some $f\in \mathbb{Q}[p_{1},p_{2},\ldots]$, we set
\[
F[x]=f(p_{1}(x),p_{2}(x),\ldots).
\]
The plethystic action satisfies the following properties:
\[
     (FG)[x]=F[x]G[x],\quad (F+G)[x]=F[x]+G[x],\quad r[x]=r, \quad F,G\in\text{Sym}_{\mathbb{Q}}[X], r\in \mathbb{Q}, x\in\Lambda.
\]
For each $x\in\Lambda$, the plethystic action $F\mapsto F[x]$
gives a homomorphism of $\mathbb{Q}$-algebras from $\text{Sym}_{\mathbb{Q}}[X]$ to $\Lambda$.
\end{definition}
\begin{definition}
Let $R$ be a base ring containing $\mathbb{Q}$ and let $\Lambda$ be a topological lambda ring. For $x\in\Lambda$, define $\Exp[x]$ as:
%Define $\Exp[x]$ for any 
%$x$ in any topological lambda ring %$\Lambda$ containing $R$ 
%as follows:
\[
\Exp[x]=\exp\bigg(\sum_{n=1}^{\infty}\frac{p_{n}[x]}{n}\bigg)
\]
provided that the right hand side converges in the topology of $\Lambda$.
\end{definition}
\subsection{Counting vector bundles over $\mathbb{P}^1$ with nilpotent endomorphisms preserving flags at $0$ and $\infty$}
Let $\Xi_{n}:=\{e_{1}-e_{2},\ldots,e_{n-1}-e_{n}\}$ denote the set of simple roots of $GL_n$ relative to the diagonal torus $T_n$ and the Borel subgroup $B_n$ consisting of upper-triangular matrices. Consider the standard full flag $E_{\bullet}=\{E_{j}\}$ in $\mathbb{F}_{q}^{n}$. Let $J\subset\Xi_{n}$. Recall from Section \ref{parabolic subgroups} that $P_{J}$ denotes the standard parabolic subgroup of $GL_{n}$ corresponding to the subset $J$. Then $P_{J}$ is the stabilizer in $GL_{n}$ of the flag obtained by removing from $E_{\bullet}$ the terms $E_{j}$ for $e_{j}-e_{j+1}\in J$. From now on, we identify $\mathcal{P}(\Xi_{n})$ with the set of standard parabolic subgroups of $GL_{n}$ via $J\mapsto P_{J}$. 

Let $\Pi_{n}$ denote the set of partitions of $\{1,\cdots,n\}$. For any partition
$\nu=(\nu_{1}\geq\nu_{2}\geq\ldots\geq\nu_{l})\in\Pi_n$,
set 
\[
J(\nu):=\{e_{i}-e_{i+1}:i\neq\nu_{1},\nu_{1}+\nu_{2},\ldots,\nu_{1}+\nu_{2}+\ldots+\nu_{l}=n,  1\leq i\leq n-1\}.
\]
This gives an inclusion from $\Pi_{n}$ to $\mathcal{P}(\Xi_{n})$, $\nu\mapsto P_{J(\nu)}$, where the image consists of stabilizers in $GL_{n}$ of standard partial flags with jumps given by partitions. If we compose this map with the map that associates to each standard parabolic subgroup its Levi factor, then we get a bijection between the set of partitions of $n$ and $GL_{n}(\mathbb{F}_{q})$-conjugacy classes of Levi $\mathbb{F}_{q}$-subgroups of $GL_{n}$.
Define $\mu(\nu):\mathbb{G}_{m}\rightarrow T_n$ as:
\[
t\mapsto\diag(\overbrace{t,\ldots,t}^{\nu_{1}\text{ times}},\ldots,\overbrace{t^{l},\ldots,t^{l}}^{\nu_{l}\text{ times}}).
\]
Recall $L_{\mu}$ from Section \ref{main theorem}. We set $L_{\nu}:=L_{\mu(\nu)}$. Notice that we have $L_{\nu}\cong GL_{\nu_{1}}\times\ldots\times GL_{\nu_{l}}$ (see Section \ref{main theorem}).

For $\lambda\in\Pi_{n}$, we will denote the monomial symmetric function corresponding to $\lambda$ by $m_{\lambda} $ and the homogeneous symmetric function corresponding to $\lambda$ by $h_{\lambda}$. Before proceeding, we make the following convention.
\begin{convention}\label{symmetric with functions}
We identify symmetric functions of degree $n$
with the associate invariant functions on $\mathcal{P}(\Xi_{n})$ by identifying $m_{\lambda}$ with $\delta_{[J(\lambda)]}$, $\lambda\in\Pi_{n}$.
\end{convention}
Let $\mu\in X_{+}(T_{n})$. Recall from Section \ref{main theorem} that $[\mathcal{T}rip_{\mu}]$ is the function on $\mathcal{P}(\Xi_{n})\times\mathcal{P}(\Xi_{n})$ that counts the number of $\mathbb{F}_{q}$-points of $\mathcal{T}rip_{\mu}(J_{0},J_{\infty})$, $(J_{0},J_{\infty})\in\mathcal{P}(\Xi_{n})\times\mathcal{P}(\Xi_{n})$. Since $[St_{L_{\mu}}]$ is an associate invariant function (Corollary~\ref{sp and st are associate invariant}) and $\pi_{\mu}=\Delta_{GL_{n}}(\Pi_{\mu},\cdot)$ (Remark \ref{pi=delta}(ii)), $(\pi_{\mu}\otimes\pi_{\mu})([St_{L_{\mu}}])$ is associate invariant by Corollary~\ref{sp and st are associate invariant}. Now using Corollary~\ref{final},
we consider $[\mathcal{T}rip_{\mu}]$ as a symmetric function (see Convention~\ref{symmetric with functions}). Thus we can write $[\mathcal{T}rip_{\mu}]$ as: 
%Since $\pi_{\mu}=\Delta_{GL_{n}}(\Pi_{\mu},\cdot)$ (see Remark~\ref{pi=delta}(iii)), by Corollary~\ref{final}, Lemma~\ref{associateinvariantdelta} and Corollary~\ref{sp and st are associate invariant},
%we consider $[\mathcal{T}rip_{\mu}]$ as a symmetric function (see Convention~\ref{symmetric with functions}). Thus we can write $[\mathcal{T}rip_{\mu}]$ as: 
\[
[\mathcal{T}rip_{\mu}]=\sum_{(\nu^{0},\nu^{\infty})\in\Pi_{n}\times\Pi_{n}}|\mathcal{T}rip_{\mu}(J(\nu^{0}),J(\nu^{\infty}))|m_{\nu^{0}}(X)m_{\nu^{\infty}}(Y).
\]
For $\mu\in X_{+}(T_{n})$, define the symmetric function $C_{\mu}[X,Y;q]$ as:
\[
C_{\mu}[X,Y;q]=\frac{[\mathcal{T}rip_{\mu}]}{|\aut(\cE_{\mu})|}
\]
and consider the following element in $\text{Sym}_{\mathbb{Q}(q)}[X,Y][[t]]$:
\[
\Omega_{n,(0,\infty)}^{\leq 0}(\mathbb{P}^{1})[X,Y;q,t]
=
\displaystyle\sum_{\mu\in X_{+}(T_{n})^{*}}t^{-\deg(\cE_{\mu})}C_{\mu}[X,Y;q],
\]
where $X_{+}(T_{n})^{*}$ consists of cocharaters $\mu:\mathbb{G}_{m}\rightarrow T_{n}$ of the form
\[
t\mapsto\diag(\overbrace{t^{-d_{1}},\ldots,t^{-d_{1}}}^{\mu_{1}\text{ times}},\ldots,\overbrace{t^{-d_{m}},\ldots,t^{-d_{m}}}^{\mu_{m}\text{ times}}),\quad 0\leq d_{1}<\ldots<d_{m}, \mu_{i}>0, 1\leq i\leq m, \sum_{i=1}^{m}\mu_{i}=n.
\]
Explicitly,
\[
\Omega_{n,(0,\infty)}^{\leq 0}(\mathbb{P}^{1})[X,Y;q,t]
=
\displaystyle\sum_{\mu\in X_{+}(T_{n})^{*}}t^{-\deg(\cE_{\mu})}\sum_{(\nu^{0},\nu^{\infty})\in\Pi_{n}\times\Pi_{n}}\frac{|\mathcal{T}rip_{\mu}(J(\nu^{0}),J(\nu^{\infty}))|}{|\aut(\cE_{\mu})|}m_{\nu^{0}}(X)m_{\nu^{\infty}}(Y).
\]
Notice that $\Omega_{n,(0,\infty)}^{\leq 0}(\mathbb{P}^{1})[X,Y;q,t]$ defined above is the same as the one considered in \cite[Section 5.4]{Mel}.
Now using our techniques, we would like to re-derive the following result of Mellit \cite[Section 5.4]{Mel}.
\begin{proposition}\label{mellit's result}
The following holds as formal series in $t$ with coefficients in the completion of $\text{Sym}_{\mathbb{Q}(q)}[X,Y]$:
\[
\sum_{n=0}^{\infty}\Omega^{\leq 0}_{n,(0,\infty)}(\mathbb{P}^{1})[X,Y;q,t]
=
\Exp\Bigg[\frac{XY}{(q-1)(1-t)}\Bigg], \quad \text{ where } XY=\sum_{i,j}x_{i}y_{j}.
\]
\end{proposition}
Recall from Section \ref{coproduct} that $[Sp_{GL_{n}}]$ is the function on $\mathcal{P}(\Xi_{n})$ that counts the number of $\mathbb{F}_{q}$-points of $Sp_{GL_{n}}(J)$, $J\in\mathcal{P}(\Xi_{n})$. Using Corollary \ref{sp and st are associate invariant} and Convention \ref{symmetric with functions}, we consider $[Sp_{GL_{n}}]$ as a symmetric function.

As a first step in proving Proposition \ref{mellit's result}, we prove the following:
\begin{proposition}\label{sp=hn}
The following holds in $\text{Sym}_{\mathbb{Q}(q)}[X]$:
\[
h_{n}\left[\frac{X}{q-1}\right]=\frac{1}{|GL_{n}|}[Sp_{GL_{n}}].
\]
\end{proposition}
\begin{proof}
By $\eqref{for GLn}$, the desired equality can be rewritten as:
\begin{equation}\label{0}
h_{n}\left[\frac{X}{q-1}\right]=\displaystyle\sum_{\nu\in\Pi_{n}}\frac{q^{\text{dim}(L_{\nu})}}{q^{n}|L_{\nu}|}m_{\nu}(X).
\end{equation}
We have the following idenitity(see \cite[Chapter 4, Section 2]{Mac}) in $\text{Sym}_{\mathbb{Q}(q)}[X,Y]$:
\begin{equation}\label{3}
    h_n(XY)=\sum_{\nu\in\Pi_{n}}m_{\nu}(X)h_{\nu}(Y).
\end{equation}
Then the specialization $x_{i}\mapsto x_{i},y_{j}\mapsto q^{-(j-1)},i,j\in\mathbb{N}$ gives a homomorphism of lambda rings $\text{Sym}_{\mathbb{Q}(q)}[X,Y]\rightarrow \Sym_{\mathbb{Q}[[q^{-1}]]}[X]$. Thus this specialization commutes with the plethystic action and we have
\[
h_{n}\left[X\bigg(1+\frac{1}{q}+\cdots\frac{1}{q^{j}}+\cdots\bigg)\right]
=
\sum_{\nu\in\Pi_{n}}m_{\nu}(X)h_{\nu}\left[1+\frac{1}{q}+\cdots\frac{1}{q^{j}}+\cdots\right]\quad \text{in } \text{Sym}_{\mathbb{Q}[[q^{-1}]]}[X]
\]
\[
h_{n}\left[\frac{qX}{q-1}\right]
=
\sum_{\nu\in\Pi_{n}}m_{\nu}(X)h_{\nu}\left[\frac{q}{q-1}\right]\quad \text{in } \text{Sym}_{\mathbb{Q}[[q^{-1}]]}[X].
\]
Since the terms of the above identity lie in $\text{Sym}_{\mathbb{Q}(q)}[X]$, the equality holds in $\text{Sym}_{\mathbb{Q}(q)}[X]$.

Since $h_{\nu}[qA]=q^{|\nu|}h_{\nu}[A]$ for any $A\in\text{Sym}_{\mathbb{Q}(q)}[X]$, we get
\begin{equation}\label{new equation}
h_{n}\left[\frac{X}{q-1}\right]
=
\sum_{\nu\in\Pi_{n}}m_{\nu}(X)h_{\nu}\left[\frac{1}{q-1}\right].
\end{equation}
Now we need to calculate $h_{\nu}\left[\frac{1}{q-1}\right]$, this follows from the following lemma:
\begin{lemma}\label{hnq}
The following holds in $\mathbb{Q}(q)$:
\[
{h_{n}\left[\frac{1}{1-q}\right]=\frac{1}{(1-q)(1-q^2)\cdots(1-q^n)}}.
\]
\end{lemma}
\begin{proof}
Let $H(w):=\sum_{r\geq 0}h_{r}(X)w^{r}\in(\text{Sym}_{\mathbb{Q}(q)}[X])[[w]]$ be the generating function for the homogeneous symmetric functions and let $P(w):=\sum_{r\geq 1}p_{r}(X)w^{r-1}\in(\text{Sym}_{\mathbb{Q}(q)}[X])[[w]]$ be the generating function for the power sum symmetric functions. The lemma follows from \cite[Chapter I, Section 2, Example 4]{Mac} and the following well-known identity in $(\text{Sym}_{\mathbb{Q}(q)}[X])[[w]]$:
\[
H(w)=\exp\left(\int P(w)dw\right).
\]
%Now,
%\begin{equation}\label{hn}
%P\left[\frac{1}{1-q}\right]=\sum_{r\geq 1}p_{r}\left[\frac{1}{1-q}\right]w^{r-1}
%=\sum_{r\geq 1}w^{r-1}\bigg(\sum_{m\geq 0}(q^r)^m\bigg)
%=\sum_{m\geq 0}\frac{q^m}{1-wq^m}.
%\end{equation}
%By $\eqref{hn}$ we have
%\[
%H\left[\frac{1}{1-q}\right]=\exp\left(\int P\left[\frac{1}{1-q}\right]dw\right)
%=
%\exp\bigg(\int \sum_{m\geq 0}\frac{q^m}{1-wq^m}dw\bigg)=\prod_{m\geq 0}\frac{1}{1-wq^m}.
%\]
%Now the lemma follows from \cite[Chapter I, Section 2, Example 4]{Mac}.
\end{proof}
We return to the proof of Proposition \ref{sp=hn}. The specialization $q\mapsto 1/q, x_{i}\mapsto x_{i}, i\in\mathbb{N}$ gives an automorphism of lambda rings $\text{Sym}_{\mathbb{Q}(q)}[X]\rightarrow\text{Sym}_{\mathbb{Q}(q)}[X]$. Thus, this specialization commutes with the plethystic action on $\Sym_{\mathbb{Q}(q)}[X]$ and we have
\[
h_{n}\left[\frac{1}{1-\frac{1}{q}}\right]=\frac{1}{(1-\frac{1}{q})(1-\frac{1}{q^2})\cdots(1-\frac{1}{q^n})}
=
\frac{q q^2\cdots q^n}{(q-1)(q^2-1)\cdots(q^n-1)}
\]
Since $h_n[qA]=q^nh_{n}[A]$ for any $A\in\text{Sym}_{\mathbb{Q}(q)}[X]$, we get
\[
h_{n}\left[\frac{1}{q-1}\right]=\frac{1}{q^n}\frac{q q^2\cdots q^n}{(q-1)(q^2-1)\cdots(q^n-1)}.
\]
Now \eqref{new equation} gives
\begin{equation}\label{5}
h_{n}\left[\frac{X}{q-1}\right]
=
\sum_{\nu\in\Pi_{n}}m_{\nu}(X)\prod_{i=1}^{k}h_{\nu_i}\left[\frac{1}{q-1}\right]
=
\sum_{\nu\in\Pi_{n}}m_{\nu}(X)\frac{1}{q^n}\frac{\prod_{i=1}^{k}q q^2\cdots q^{\nu_i}}{\prod_{i=1}^{k}(q-1)(q^2-1)\cdots(q^{\nu_i}-1)}.
\end{equation}
The coefficient of $m_{\nu}(X)$ in equation \eqref{5} is equal to
\[
\frac{1}{q^n}\frac{\prod_{i=1}^{k}(q q^2\cdots q^{\nu_i})^{2}}{\prod_{i=1}^{k}q^{\nu_{i}}(q^{\nu_{i}}-q^{\nu_{i}-1})(q^{\nu_i}-q^{\nu_{i}-2})\cdots(q^{\nu_i}-1)}=\frac{1}{q^n}\frac{q^{\sum \nu_{i}^{2}}q^{\sum \nu_i}}{q^{n}\prod_{i=1}^{k}|GL_{\nu_{i}}(\mathbb{F}_{q})|}=\frac{1}{q^n}\frac{q^{\sum \nu_{i}^{2}}}{\prod_{i=1}^{k}|GL_{\nu_{i}}(\mathbb{F}_{q})|}.
\]
Since $L_{\nu}\cong GL_{\nu_{1}}\times\ldots\times GL_{\nu_{l}}$, Proposition \ref{sp=hn} follows.
\end{proof}
Next, consider one of the two standard coproduct on symmetric functions:
\[
    \Delta^{n}:\Sym^{n}_{\mathbb{Z}}[X] \rightarrow\Sym^{n}_{\mathbb{Z}}[X]\otimes\Sym^{n}_{\mathbb{Z}}[Y]=\Sym^{n}_{\mathbb{Z}}[X,Y], \quad f(X)\mapsto f(XY).
\]
Let $\Delta'_{GL_{n}}$ denote the restriction of $\Delta_{GL_{n}}$ to the associate invariant functions. We would like to show that $\Delta^n$ agrees with $\Delta'_{GL_{n}}$ by identifying symmetric functions of degree $n$ with the associate invariant functions on $\mathcal{P}(\Xi_{n})$ (see Convention \ref{symmetric with functions}).
First we need a notation.
\begin{notation}
For any sequence of positive integers $\alpha=(\alpha_{1},\ldots,\alpha_{m})$ such that $\sum_{j=1}^{m}\alpha_{j}=n$, we will denote the subgroup $S_{\alpha_{1}}\times\ldots\times S_{\alpha_{m}}$ of $S_{n}$ by $S_\alpha$.
\end{notation}
We have the following proposition.
\begin{proposition}\label{coproduct agrees with coproduct}
We have
$
\Delta'_{GL_{n}}=\Delta^n.
$
\end{proposition}
\begin{proof}
Since $m_{\nu}$, $\nu\in\Pi_{n}$ form a basis of $\Sym_{\mathbb{Z}}^{n}[X]$, it is enough to check that $\Delta'_{GL_{n}}$ agrees with $\Delta^{n}$ on this basis.
We re-write the conclusion of Lemma \ref{associateinvariantdelta} for $GL_n$. Let $\nu$ be a partition of $n$. It gives an equivalence relation $\sim_\nu$ on $\{1,\ldots,n\}$, where $i\sim_\nu j$ if and only if there exists a $t$ such that $\nu_1+\ldots+\nu_t\leq i,j<\nu_1+\ldots+\nu_{t+1}$. Note that $S_n$ acts on $\{1,\ldots,n\}$ and thus on the equivalence relations. For an equivalence relation $\sim$ on $\{1,\ldots,n\}$, we will write $Part(\sim)\in\Pi_n$ for the corresponding partition of $n$, that is, the ordered sequence of sizes of equivalence classes. 
By Lemma \ref{associateinvariantdelta}, we have
\[
    \Delta'_{GL_{n}}(m_\nu)=\sum_{\lambda,\mu\in\Pi_{n}}n_{\nu}^{\lambda,\mu}m_\lambda\otimes m_\mu,
\]
where 
\[
n_{\nu}^{\lambda,\mu}=\big|\{w\in S_{\lambda}\backslash S_{n}/S_{\mu}:Part(\sim_{\lambda}\cap w(\sim_{\mu}))=\nu\}\big|.
\]
%If we identify roots for $GL_n$ with pairs of integers, then $\Phi(J_\nu)$ is equal to $\sim_\nu$ without the diagonal.
Thus we have
\begin{equation}\label{eq:GLn}
  n_\nu^{\lambda,\mu}
  =
  \sum_{\substack{w\in S_n\colon
  \\
  Part(\sim_\lambda\cap w(\sim_\mu))=\nu}}\frac{|w^{-1}S_{\lambda}w\cap S_{\mu}|}{|S_\mu||S_\lambda|}.
\end{equation}
Now consider the coproduct $\Delta^n$.
We have
\[
    \Delta^{n}(m_\nu)=\sum_{[(i_1,j_1),\ldots,(i_n,j_n)]}(X_{i_1}Y_{j_1})\ldots(X_{i_n}Y_{j_n})
\]
where the sum is over all multisets $[(i_t,j_t)]$ such that the multiplicities of elements are given by $\nu$.

The group $S_n$ is acting naturally on length $n$ sequences. Let $(\mu)$ be the standard sequence
\[
    \underbrace{1,\ldots,1}_{\mu_1\text{ times}},\underbrace{2,\ldots,2}_{\mu_2\text{ times}},\ldots.
\]
In $\Delta^{n}(m_{\nu})$, $X^\lambda Y^\mu$ occurs as:
\[
    \sum_{j_1,\ldots,j_n}\frac{1}{|\text{orbit of }S_\lambda\text{ on }j_{1},\ldots,j_{n}|}(X_1Y_{j_1}\ldots X_1Y_{j_{\lambda_1}})(X_2Y_{j_{\lambda_1+1}}\ldots X_2Y_{j_{\lambda_2}})\ldots,
\]
where the summation is over all sequences $j_1,\ldots,j_n$ such that $[j_{1},\ldots,j_{n}]=[(\mu)]$ and $Part(\sim_\lambda\cap\sim_j)=\nu$, where $\sim_j$ denotes the equivalence relation $t\sim_j s$ iff $j_t=j_s$.
Let $\widetilde{n}_{\nu}^{\lambda,\mu}$ denote the coefficient of $m_{\lambda}\otimes m_{\mu}$ in $\Delta^{n}(m_{\nu})$.
The condition $[j_1,\ldots,j_n]=[(\mu)]$ is equivalent to the existence of $w\in S_n$ such that $w\cdot(\mu)=(j_1,\ldots,j_n)$, in which case $w(\sim_\mu)=\sim_j$. Since there are exactly $|S_\mu|$ such $w$, we get
\[
\widetilde{n}_\nu^{\lambda,\mu}
  =
  \sum_{\substack{w\in S_n\colon
  \\
  Part(\sim_\lambda\cap w(\sim_\mu))=\nu}}\frac{|w^{-1}S_{\lambda}w\cap S_{\mu}|}{|S_\mu||S_\lambda|},
\]
which agrees with~\eqref{eq:GLn}.
This finshes the proof of Proposition \ref{coproduct agrees with coproduct}.
\end{proof}
Recall the vector bundle $
\cE$ of rank $n$ over $\mathbb{P}^1$ in \cite[Section 5.4]{Mel}, which is defined as:
\[
\cE=\mathcal{O}(-d_{1})^{\mu_{1}}\oplus...\oplus\mathcal{O}(-d_{m})^{\mu_{m}},\quad 0\leq d_{1}<\ldots<d_{m},\quad \mu_{i}>0,\quad 1\leq i\leq m, \quad \sum_{i=1}^{m}\mu_{i}=n.
\]
Let $\mu:\mathbb{G}_{m}\rightarrow T_{n}$ be the cocharacter of the form 
\[
t\mapsto\diag(\overbrace{t^{-d_{1}},\ldots,t^{-d_{1}}}^{\mu_{1}\text{ times}},\ldots,\overbrace{t^{-d_{m}},\ldots,t^{-d_{m}}}^{\mu_{m}\text{ times}}),\quad 0\leq d_{1}<\ldots<d_{m},\quad \mu_{i}>0,\quad 1\leq i\leq m, \quad \sum_{i=1}^{m}\mu_{i}=n.
\]
Then we have $\mu\in X_{+}(T_{n})$ and we get that $\cE=\cE_{\mu}$ (see \ref{basic_principal bundles}).

Let us write $\mu=(\widetilde{\mu}_{1},\ldots\widetilde{\mu}_{m})$, where $\widetilde{\mu}_{k}:\mathbb{G}_{m}\rightarrow T_{\mu_{k}}$, $1\leq k\leq m$ is the cocharacter
\[
t\mapsto\diag(\overbrace{t^{-d_{k}},\ldots,t^{-d_{k}}}^{\mu_{k}\text{ times}}).
\]
We have $\widetilde{\mu}_{k}\in X_{+}(T_{\mu_{k}})$. 
The following is a key factorization result, which is a corollary of Theorem \ref{trip}:
\begin{corollary}\label{factorization}
For the vector bundle $\cE$ over $\mathbb{P}^1$, we have 
\[
C_{\mu}[X,Y;q]=\prod_{k=1}^{m}h_{\mu_{k}}\bigg[\frac{XY}{q-1}\bigg].
%\displaystyle\prod_{k=1}^{m}C_{\widetilde{\mu}_{k}}[X,Y;q].
\]
\end{corollary}
\begin{proof}
%Let $f_k$ be an associate invariant function on $\mathcal{P}(\Pi_k)$, where $\Pi_k$ is the set of simple roots of $GL_{\mu_k}$, $1\leq k\leq m$. According to our Convention~\ref{symmetric with functions}, $f_k$ is viewed as an element of $\Sym_{\mathbb{Z}}^{\mu_{k}}[X]$. However, we can also view $f_k$ as a symmetric polynomial $f'_k$ in the variables $x_{\mu_{1}+\ldots+\mu_{k-1}+1}$, \ldots, $x_{\mu_{1}+\ldots+\mu_k}$. Recall that in Section \ref{main theorem} we defined the map $\pi_\mu$.
%$\colon\mathbb{Z}[\mathcal{P}(\Pi_{\mu})]\rightarrow \mathbb{Z}[\mathcal{P}(\Pi)]$.
%In the case of $GL_n$, this map relates products for the two different interpretations of the associate invariant functions $f_{k}$, $1\leq k\leq m$:
Let $f_k$ be an associate invariant function on $\mathcal{P}(\Xi_{\mu_{k}})$, where $\Xi_{\mu_{k}}$ is the set of simple roots of $GL_{\mu_k}$, $1\leq k\leq m$. According to our Convention~\ref{symmetric with functions}, $f_k$ is viewed as an element of $\Sym_{\mathbb{Z}}^{\mu_{k}}[X]$. However, we can also view $f_k$ as a symmetric polynomial $f'_k$ in the variables $x_{\mu_{1}+\ldots+\mu_{k-1}+1}$, \ldots, $x_{\mu_{1}+\ldots+\mu_k}$. Recall the map $\pi_{\mu}$ defined in Section \ref{main theorem}.
%Recall that in Section \ref{main theorem} we defined the map $\pi_\mu$.
%$\colon\mathbb{Z}[\mathcal{P}(\Pi_{\mu})]\rightarrow \mathbb{Z}[\mathcal{P}(\Pi)]$.
%In the case of $GL_n$, this map relates products for the two different interpretations of the associate invariant functions $f_{k}$, $1\leq k\leq m$:
In the case of $GL_n$, this map relates products for the two different interpretations of the associate invariant functions $f_{k}$, $1\leq k\leq m$ in the following way:
\begin{lemma}\label{lm2}
%Keep notations as above. Then
We have
\[
    f_1\ldots f_k=\pi_\mu(f'_1\ldots f'_k).
\]    
\end{lemma}
\begin{proof}
  In the case of $GL_n$ the map $\pi_\mu$ is the symmetrization map.
\end{proof}
We return to the proof of Corollary~\ref{factorization}. Recall from Section \ref{coproduct} that $[St_{GL_{\mu_{k}}}]$ is the function on $\mathcal{P}(\Xi_{\mu_{k}})\times\mathcal{P}(\Xi_{\mu_{k}})$ that counts the number of $\mathbb{F}_{q}$-points of $St_{GL_{\mu_{k}}}(J_{1},J_{2})$, $(J_{1},J_{2})\in\mathcal{P}(\Xi_{\mu_{k}})\times\mathcal{P}(\Xi_{\mu_{k}})$. Using Corollary \ref{sp and st are associate invariant}, we consider $[St_{GL_{\mu_{k}}}]$ as a symmetric function (see Convention~\ref{symmetric with functions}).
We can also view $[St_{GL_{\mu_k}}]$ as a symmetric function $[St_{GL_{\mu_k}}]'$ in the variables $x_{\mu_{1}+\ldots+\mu_{k-1}+1}$, \ldots, $x_{\mu_{1}+\ldots+\mu_k}$. 
Now by Corollary~\ref{final}, Fact~\ref{aut of emu} and Lemma~\ref{lm1}, we have
\[
    C_\mu[X,Y;q]=(\pi_\mu\otimes\pi_\mu)\big([St_\mu]\big)\big/\prod_{k}|GL_{\mu_{k}}|=(\pi_\mu\otimes\pi_\mu)\Big(\prod_k[St_{\mu_k}]'\Big)\big/\prod_{k}|GL_{\mu_{k}}|.
\]
Using Lemma~\ref{lm2} in each variable, we get
\[
    C_\mu[X,Y;q]=
%(\pi_\mu\otimes\pi_\mu)\Big(\prod_k[St_{\mu_k}]'\Big)%\big/\prod_{k}|GL_{\mu_{k}}|=
    \prod_k[St_{\mu_k}]/|GL_{\mu_{k}}|
    =
\prod_{k=1}^{m}\frac{\Delta_{GL_{\mu_k}}\big([Sp_{GL_{\mu_{k}}}]\big)}{|GL_{\mu_{k}}|},
   % =\prod_k C_{\tilde\mu_k}[X,Y;q],
\]
where the last equality follows from Theorem \ref{steinberg} $(ii)$. Now the corollary follows from Proposition \ref{sp=hn} and Proposition \ref{coproduct agrees with coproduct}.
\end{proof}
Now we are ready to prove Proposition \ref{mellit's result}. We have 
\[
  \displaystyle\sum_{n=0}^{\infty}\Omega_{n,(0,\infty)}^{\leq 0}(\mathbb{P}^{1})[X,Y;q,t]
  =
  \sum_{n=0}^{\infty}\displaystyle\sum_{\mu\in X_{+}(T_{n})^{*}}t^{-\deg(\cE_{\mu})}C_{\mu}[X,Y;q]
  =
  \sum_{n=0}^{\infty}\displaystyle\sum_{\substack{\mu=(\mu_{1},\ldots,\mu_{m}),\\d=(0\leq d_{1}<d_{2}<\ldots< d_{m}):\\\sum\mu_{k}=n}}t^{\sum_{k=1}^{m}d_{k}\mu_{k}}h_{\mu}\Bigg[\frac{XY}{q-1}\Bigg],
\]
where the second equality follows from Corollary \ref{factorization}. The above is equal to
\[
  \displaystyle\prod_{d=0}^{\infty}\displaystyle\sum_{k=0}^{\infty}t^{dk}h_{k}\Bigg[\frac{XY}{q-1}\Bigg]
  =
  \displaystyle\prod_{d=0}^{\infty}\Exp\Bigg[\frac{t^{d}XY}{q-1}\Bigg]
  =
  \Exp\Bigg[\frac{XY}{q-1}\sum_{d=0}^{\infty}t^{d}\Bigg]
  =
  \textrm{Exp}\Bigg[\frac{XY}{(q-1)(1-t)}\Bigg].
\]
This finishes the proof of Proposition \ref{mellit's result}.
\qed
%\begin{remark}\label{stability}
%\begin{enumerate}[(i)]
%    \item It 
%\end{enumerate}
%\end{remark}

\section{Conflict of interest} The author has no conflict of interest to declare that are relevant to this article.

\Address

\begin{thebibliography}{9}
\bibitem{Bo}
A. Borel.
\textit{Linear Algebraic Groups}.
Second enlarged edition, Springer.

\bibitem{BT}
A. Borel, J. Tits.
\textit{Compléments à l’article: Groupes réductifs}.
Publications mathématiques de l’IHÉS., tome 41 (1972), p. 253-276.

\bibitem{Bro} 
P. Brosnan. 
\textit{On motivic decompositions arising from the method of Bialynicki--Birula}. 
Inventiones mathematicae, 161: 91-111, 2005.

\bibitem{Car}
R. Carter.
\textit{Finite groups of
Lie type: Conjugacy classes and complex characters}
Wiley Classics Library Edition, 1993.

\bibitem{CGP}
B. Conrad, O. Gabber, G. Prasad. 
\textit{Pseudo-reductive groups}.
Cambridge University Press, June 2015.

%\bibitem{Con1}
%B. Conrad. 
%\textit{Reductive groups over fields (online notes)}. 
%\url{http://virtualmath1.stanford.edu/~conrad/249BW16Page/handouts/249B_2016.pdf}.
%\bibitem{Con2} 
%B. Conrad. 
%\textit{Reductive group schemes (online notes)}.
%\url{http://math.stanford.edu/~conrad/papers/luminysga3.pdf}.
%\bibitem{Con3} 
%B. Conrad. 
%\textit{Online course notes}.
%\url{http://virtualmath1.stanford.edu/~conrad/249BW16Page/handouts/alggroups.pdf}
%\bibitem{Con4} 
%B. Conrad. 
%\textit{Online course notes on Dynamic approach to algebraic groups}.
%\url{http://virtualmath1.stanford.edu/~conrad/252Page/handouts/dynamic.pdf}
%\bibitem{Con5}
%B. Conrad.
%\textit{Online course notes on Tits systems}.
%\url{http://virtualmath1.stanford.edu/~conrad/249BW16Page/handouts/titssystem.pdf}
\bibitem{CR}
C. W. Curtis, I. Reiner.
\textit{Methods of representation theory}.
Wiley-interscience publication, Volume II.

\bibitem{DM}
F. Digne, J. Michel. 
\textit{Representations of finite groups of Lie type}. 
Cambridge University Press 1991.

\bibitem{Fed2}
R. Fedorov.
\textit{Affine Grassmannians of group schemes and exotic principal bundles over $\mathbb{A}^1$}.
American Journal of Mathematics,
Volume 138, Number 4, August 2016,
pp. 879-906

\bibitem{Fre} 
L. Fresse. 
\textit{Existence of affine pavings for varieties of partial flags associated to nilpotent elements}.
International Mathematics Research Notices 418-472,2016.

\bibitem{Gi}
P. Gille.
\textit{Torseurs sur la droite affine}.
Transformation Groups  7 (2002), 231-245.

\bibitem{La} 
S. Lang. 
\textit{Algebraic groups over finite fields}. Amer. J. Math., 78:555–563, 1956.

\bibitem{Mac}
I.G. Macdonald.
\textit{Symmetric Functions and Hall Polynomials}. Oxford Science Publications, Second Edition, 1995.

\bibitem{Mel} 
A. Mellit. 
\textit{Poincare polynomials of character varieties, Macdonald polynomials and affine Springer fibers}.
Ann. of Math, 192(1):165-228, 2020.

\bibitem{Mil1}
J. Milne.
\textit{Algebraic groups: The theory of group schemes of finite type over a field}.
Cambridge University Press, First Edition, 2017.

\bibitem{MS20}
S. Mozgovoy and O. Schiffmann.
\textit{Counting Higgs bundles and type A quiver bundles}.
Compositio Math. 156 (2020), 744–769.

\bibitem{Poo}
B. Poonen.
\textit{Rational points on varieties}.
American Mathematical Society.

\bibitem{Ra} 
A. Ramanathan. 
\textit{Deformations of principal bundles on the projective line}.
Inventiones mathematicae, 71:165–191, 1983.

\bibitem{Ric}
R. W. Richardson.
\textit{Intersections of double cosets in algebraic groups}.
Indag. Mathem., N.S., 3 (I), 69-77.

\bibitem{Sch} 
O. Schiffmann. 
\textit{Indecomposable vector bundles and stable Higgs bundles over smooth projective curves}.
Ann. of Math, (2)
183(1):297–362, 2016.

\bibitem{Spr}
T. A. Springer.
\textit{Linear Algebraic Groups}.
Second edition, Birkhauser.

\bibitem{Spr2}
T. A. Springer.
\textit{The Steinberg Function of a Finite Lie Algebra}.
Inventiones mathematicae, 58:211–215, 1980.

\bibitem{Ste}
Robert Steinberg.
\textit{Regular elements of semisimple algebraic groups}.
Inst. Hautes Etudes Sci. Publ. Math. 25 (1965), 49-80.
\end{thebibliography}
\end{document}